\documentclass[12pt]{amsart}
\usepackage{amssymb}
\usepackage{bbm}
\usepackage{verbatim}
\usepackage{graphicx,color}
\usepackage{enumerate}
\usepackage[all]{xy}
\usepackage{tabu,makecell}

\usepackage[hyperindex=true,plainpages=false,colorlinks=false,pdfpagelabels]{hyperref}

\theoremstyle{plain}
\newtheorem{The}{Theorem}
\newtheorem*{The*}{Theorem}

\newtheorem{Lem}{Lemma}
\newtheorem{Cor}{Corollary}

\newtheorem*{Cor*}{Corollary}

\theoremstyle{definition}
\newtheorem{Def}{Definition}

\newtheorem*{Rem*}{Remark}
\newtheorem*{Con*}{Conjecture}

\DeclareMathOperator{\End}{End}

\DeclareMathOperator{\Diff}{Diff}
\DeclareMathOperator{\Tr}{tr}             

\DeclareMathOperator{\Id}{id}

\DeclareMathOperator{\Ad}{Ad}
\DeclareMathOperator{\ad}{ad}  

\DeclareMathOperator{\mM}{\mathcal M}

\DeclareMathOperator{\Jac}{Jac}

\renewcommand{\Im}{\operatorname{Im}}
\renewcommand{\Re}{\operatorname{Re}}

\newcommand{\udot}{\!\dot{\phantom{i}}\!}

\DeclareMathOperator{\del}{\partial}

\newcommand{\R}{\mathbb{R}}

\newcommand{\C}{\mathbb{C}}
\newcommand{\N}{\mathbb{N}}
\newcommand{\Z}{\mathbb{Z}}

\newcommand{\CP}{\mathbb{CP}}

\DeclareSymbolFont{bbold}{U}{bbold}{m}{n}
\DeclareSymbolFontAlphabet{\mathbbold}{bbold}
\newcommand{\one}{\mathbbold{1}}

\graphicspath{{./figs/}}

\begin{document}

\title{Commuting Hamiltonian flows of curves in real space forms}

\author{Albert Chern}
\author{Felix Kn\"oppel}
\author{Franz Pedit}

\address{
Institute of Mathematics, MA 8-4\\
Technical University Berlin\\
Strasse des 17. Juni 136\\
10623 Berlin\\
GERMANY}

\author{Ulrich Pinkall}

\address{ Department of Mathematics and Statistics\\
University of Massachusetts\\
Amherst, MA 01003\\
 USA}



\date{\today}

\maketitle
\begin{abstract}
Starting from the vortex filament flow introduced in 1906 by Da Rios, there is a hierarchy of commuting geometric flows on space curves. The traditional approach relates those flows to the nonlinear Schr\"{o}dinger hierarchy satisfied by the complex curvature function of the space curve. 
Rather than working with this infinitesimal invariant, we describe the flows directly as vector fields on the manifold of space curves. 
This manifold carries a canonical symplectic form introduced by Marsden and Weinstein. Our flows are precisely the symplectic gradients of a natural hierarchy of invariants, beginning with length, total torsion, and elastic energy. There are a number of advantages to our geometric approach. For instance, the real part of the spectral curve is geometrically realized as the motion of the monodromy axis when varying total torsion. This insight provides a new explicit formula for the hierarchy of Hamiltonians. We also interpret the complex spectral curve in terms of curves in hyperbolic space and Darboux transforms.
Furthermore, we complete the hierarchy of Hamiltonians by adding area and volume. These allow for the characterization of elastic curves as solutions to an isoperimetric problem: elastica are the critical points of length while fixing area and volume.
\end{abstract}

\section{Introduction}

The study of curves and surfaces in differential geometry and geometric analysis has given rise to a number of important global problems, which 
play a pivotal role in the development of those subjects.  A historical example worth mentioning is Euler's study and classification of elastic planar curves, which are the critical points of the bending energy $\int \kappa^2$, the averaged squared curvature of the curve. Variational calculus, conserved quantities, and geometry combine beautifully to provide a complete solution of the problem.  Euler presumably did not know that the bending energy belongs to an infinite hierarchy of commuting energy functionals on the space of planar curves, whose commuting symplectic gradients are avatars of the modified Korteweg--de Vries (mKdV) hierarchy. 
In fact, a more complete picture arises when considering curves in 3-space, in which case the flows are a geometric manifestation of the non-linear Schr\"odinger hierarchy (of which mKdV is a reduction).

There is evidence that an analogous structure is present on the space of surfaces with abelian fundamental groups in 3- and 4-space. The energy functional in question is the Willmore energy $\int H^2$, averaging the squared mean curvature of the surface, whose critical points are Willmore surfaces.  Again, this functional is part of an infinite hierarchy of commuting functionals related to the Davey--Stewartson hierarchy in mathematical physics. The  existence of such hierarchies plays a significant role in classification problems in surface geometry including minimal, constant mean curvature, and Willmore surfaces.  

The equations mentioned, (m)KdV, non-linear Schr\"odinger, Davey--Stewartson (and some of its reductions like modified Novikov--Veselov, sinh-Gordon, etc.)~have their origins in mathematical physics and serve as prime models for {\em infinite dimensional integrable systems.} Characterizing features of these equations include soliton solutions, a transformation theory --- Darboux transforms --- with permutability properties, Lax pair descriptions on loop algebras, dressing transformations, reformulations as  families of flat connections --- zero curvature descriptions, and explicit solutions arising from linear flows on Jacobians of finite genus spectral curves. The latter provide a stratification of the hierarchy by finite dimensional classical integrable systems giving credence to the terminology. 

This note will provide a geometric inroad into these various facets of infinite dimensional integrable systems by discussing in detail the simplest example, namely the space $\mM$ of curves $\gamma$  in $\R^3$. The advantage of this example lies in its technical simplicity without loosing any of the conceptual complexity of the theory. At various junctures we shall encounter loop algebras, zero curvature equations, Jacobians etc. as useful concepts, techniques, and possible directions for further exploration. Rather than  working with closed curves $\gamma\colon S^1\to\R^3$, many of the classically relevant examples require curves $\gamma\colon M\to\R^3$ with monodromy given by an orientation preserving Euclidean motion $h(p)=Ap+a$. Such curves are equivariant 
in the sense that $\tau^*\gamma=A\gamma+a$, where $M\cong \R$ with a fixed diffeomorphism $\tau$ so that $M/\tau\cong S^1$.

The space $\mM$ has the structure of a pre-symplectic manifold thanks to the Marsden--Weinstein $2$-form $\sigma\in\Omega^2(\mM,\R)$ given by equation \eqref{eq:Arnold}. This $2$-form is degenerate along the tangent spaces to the reparametrization orbits of diffeomorphisms on $M$  compatible with the translation $\tau$. The $L^2$-metric $\langle -,-\rangle$ gives $\mM$
the additional structure of a Riemannian manifold and the two structures relate via
$\sigma(\xi,\eta)=\langle T\times \xi,\eta\rangle$
for $\xi,\eta \in T\mM$, and $T\in\Gamma(T\mM)$ the vector field whose value at $\gamma\in\mM$ is the unit length tangent  $T(\gamma)=\gamma'$. The Riemannian and pre-symplectic structures allow us to construct variational $G_E$ and  symplectic $Y_E$ gradients to a given energy functional $E\colon \mM\to\R$.  In this note we focus on the Hamiltonian aspects of the space $\mM$ and show that $\mM$ can be viewed as a  phase space for an infinite dimensional  integrable system given by a hierarchy of commuting Hamiltonians $E_k\in C^{\infty}(\mM,\R)$ and corresponding commuting symplectic vector fields $Y_k\in\Gamma(T\mM)$. 

Having introduced the space $\mM$ in Section~2, we start Section~3 with a number of classical energy functionals: 
the length functional $E_1$, 
 the total torsion functional $E_2$, measuring the turning angle of a parallel section in the normal bundle of a curve $\gamma\in\mM$ over a fundamental domain $I\subset M$ for $\tau$, and the bending, or elastic, energy $E_3$, measuring the total squared curvature of a curve over $I\subset M$.
Their variational and symplectic gradients are well known, see Table~\ref{tab:hierarchy} and also the historical Section~\ref{sec:history}, and the behavior of some of their flows has been studied from geometric analytic and Hamiltonian aspects. There are two more Hamiltonians, $E_{-1}$ and $E_{-2}$, whose significance in our context seems to be new: they arise from the flux of the infinitesimal translation and rotation vector fields on $\R^3$ through a surface spanned by the curve $\gamma\in\mM$ and the axis of its monodromy over a fundamental region $I\subset M$.  These functionals, at least for closed curves, measure a certain projected enclosed area  of $\gamma$ and the  volume of the solid torus generated by revolving $\gamma$ around a fixed axis, respectively. Since $E_{-1}$ and $E_{-2}$  are given by infinitesimal isometries, they are preserved by the symplectic flows of $E_k$ for $1\leq k\leq 3$. As an example, the vortex  filament flow
 $Y_1=\gamma'\times \gamma''$ preserves those areas and volumes as shown in Figure~\ref{fig:volumefunctional}. The functionals $E_{-1}$ and $E_{-2}$ also provide new isoperimetric characterizations of the critical points of the elastic energy $E_2$ constrained by length and total torsion, the so-called {\em Euler elastica}. For instance, an Euler elastic curve can also be described as being critical for total torsion under length and enclosed area constraints, or critical for length under enclosed area and volume constraints. 
 
 Inspecting the gradients $G_k$ and $Y_k$ of the functionals $E_k$ for $-2\leq k\leq 3$ in Table~\ref{tab:hierarchy} (we also included the vacuum  $E_0=0$), one notices the pattern
 \[
 G_{k+1}=-Y_k'
 \]
 along curves $\gamma\in\mM$. This recursion, together with the fact that $G_{k}=T\times Y_k$, can then be used to define an infinite sequence
 \cite{yasui1998, langer1999} of vector fields $Y_k\in\Gamma(T\mM)$ via
 \[
 Y_{k+1}'+T\times Y_{k}=0\,,\quad Y_0=T\,.
 \]
 Since the vortex filament flow $Y_1=\gamma'\times \gamma''$ corresponds to the non-linear Schr\"odinger flow \cite{hasimoto1972} of the associated complex curvature function $\psi$ of the curve $\gamma$, it is reasonable to expect the flows $Y_k$ to be avatars of the non-linear Schr\"odinger
 hierarchy. We regard this as evidence that the vector fields $Y_k$ commute on $\mM$ and are symplectic for a hierarchy of energy functionals
$E_k$ on $\mM$.  

Our discussion of the commutativity of the $Y_k$ in Section~4 relates the flows $Y_k$ to purely Lie theoretic flows 
$V_k$ on the loop Lie algebra $\Lambda\R^3$  of formal Laurent series with coefficients in $\R^3={\bf su}_2$. This loop Lie algebra has the decomposition 
\[
\Lambda\R^3=\Lambda^{+}\R^3\oplus\Lambda^{-}\R^3
\]
into sub Lie algebras given by positive and non-negative frequencies in $\lambda$.
The flows $V_k$ are known to commute \cite{ferus1992}. Moreover, the recursion relation for $Y_k$  corresponds to the lowest order non-trivial flow $V_0=Y\times \lambda Y_0$, a Lax pair equation, for the generating loop 
\[
Y=\sum_{k\geq 0} Y_k\lambda^{-k}
\]
on the Lie subalgebra $\Lambda^{-}\R^3\subset \Lambda\R^3$.  The main theorem in this section relates the evolution of a curve 
$\gamma\in\mM$ by the flow $Y_k$ to the evolution of the generating loop $Y$ by the Lie theoretic flow $V_k(Y)=Y\times (\lambda^{k+1}Y)_{+}$, with $(\,)_{+}$ projection along the decomposition of $\Lambda\R^3$.  This observation eventually implies the commutativity of the vector fields $Y_k$ on $\mM$.  It appears that our approach to commutativity is new. Related results in the literature  \cite{langer1999} are usually concerned with the induced flows on the space of complex curvature functions $\psi$, rather than with the flows on the geometric objects, the curves $\gamma\in\mM$, per se.  

 The connection of our flows $Y_k$ to commuting vector fields $V_k$ on the loop Lie algebra $\Lambda\R^3$ provides us with a rich solution theory  for the  dynamics generated by $Y_k$. A particularly well studied type of solutions are the finite gap curves $\gamma\in\mM$. These arise from invariant finite dimensional subspaces $\Lambda_d\subset \Lambda^{-}\R^3$ of polynomials of degree $d$ in $\lambda^{-1}$. The flows $V_k$ are non-trivial only for $k=0,\dots ,d-1$, and linearize on the (extended) Jacobian $\Jac(\Sigma)$ of a finite genus $g_{\Sigma}=d-1$ algebraic curve $\Sigma$, the {\em spectral curve} of the flows $V_k$. Thus, finite gap solutions can in principle be explicitly parametrized by theta functions on $\Sigma$. 
 
 A non-trivial problem is to single out those finite gap solutions which give rise to curves $\gamma\in\mM$ with prescribed monodromy. It seems feasible that our setting of monodromy preserving flows  $Y_k$ could contribute to this problem.  A related question, which may be within reach of our approach, is how to approximate a given curve $\gamma\in\mM$ by a finite gap solution of low spectral genus $g_{\Sigma}$. 
 
In addition to their algebro-geometric significance, finite gap solutions on $\Lambda_d$  have a variational characterization as stationary solutions $\gamma\in\mM$ to the putative functional $E_{d+1}$ constrained by the lower order Hamiltonians $E_k$, $1\leq k\leq d$. For example, Euler elastica $\gamma$ correspond to  elliptic spectral curves and hence can be explicitly parametrized by elliptic functions.
 
 In Section~5 we derive the complete list of Hamiltonians $E_k$ for which $Y_k$ are symplectic. In contrast to previous work \cite{langer1991}, where those functionals are calculated from the non-linear Schr\"odinger hierarchy \cite{faddeev1987}, we work on the space of curves $\mM$   directly. The basic ingredient comes from the geometry behind the generating loop $Y=\sum_{k\geq 0}Y_k \lambda^{-k}$ of the flows $Y_k$, namely the associated family of curves $\gamma_{\lambda}$ for $\lambda\in\R$, see Figure~\ref{fig:AssociateFamily}. These curves osculate the original curve $\gamma$ at some chosen base  point $x_0\in M$ to second order, and tend to the straight line through $\gamma(x_0)$ in direction of $\gamma'(x_0)$ as $\lambda\to \infty$, see Figure~\ref{fig:SpectralCurveCone}. The rotation monodromy ${A}_{\lambda}$ of the curves $\gamma_{\lambda}$, based at $x_0\in M$,  has axis $Y_{\lambda}(x_0)$ and we show that its rotation angle $\theta_{\lambda}$ is the Hamiltonian for the generating loop $Y$, that is, 
 \[ 
\tfrac{1}{\lambda^2} d\theta = \sigma (Y,-)\,.
\]
Applying the Gau{\ss}--Bonnet Theorem to the sector traced out by the tangent image $T_{\lambda}$ of $\gamma_{\lambda}$  over a fundamental region $I\subset M$, we derive in Theorem~\ref{Hamiltonians} an explicit expression for the angle function
 $\tfrac{1}{\lambda^2}\theta=\sum_{k\geq 0}E_k\lambda^{-k}$. As an example, $E_6$ is given by 
$E_6=\int_I (-\tfrac{1}{2}\det(\gamma',\gamma''',\gamma'''')+\tfrac{7}{8}|\gamma''|^2\det(\gamma',\gamma'',\gamma'''))\,dx$.

In  Section~6 we discuss the associated family $\gamma_{\lambda}$ for complex values of the spectral parameter $\lambda\in\C$. In this case the curves $\gamma_{\lambda}$ can be seen as curves with monodromy in hyperbolic 3-space $H^3$. 
Furthermore, the  two fixed points on the sphere at infinity of $H^3$ of the hyperbolic monodromy of $\gamma_{\lambda}$, based at $x_0\in M$,  define two Darboux transforms $\eta_{\pm}\in\mM$ of the original curve $\gamma\in\mM$, see Figure~\ref{fig:HyperbolicAssociateFamily}.  
 These  fixed points of the monodromy define a hyper-elliptic spectral curve over $\lambda\in \C$, which, at least for finite gap curves $\gamma$, is biholomorphic to the spectral curve $\Sigma$ discussed in Section~4. We have thus realized the algebro-geometric spectral curve $\Sigma$ of a finite gap space curve $\gamma\in\mM$  in terms of a family of Darboux transforms $\eta\in\mM$ parametrized by $\Sigma$.  In other words, the spectral curve $\Sigma$ has a canonical realization in Euclidean space $\R^3$ as pictured in Figures~\ref{fig:HyperbolicAssociateFamily}--\ref{fig:Darboux}. 
 
 There is a vast amount of literature pertaining to our discussion, for which we have included a final historical section. Here the reader can find a chronological exposition, with references to the relevant literature, of much of the background material. 

{\em Acknowledgments\,}: this paper arose from a lecture course the third author gave during his stay at the Yau Institute at Tsinghua University in Beijing during Spring 2018. The development of the material was also supported by SFB Transregio 109 ÒDiscretization in Geometry and DynamicsÓ at Technical University Berlin. Software support for the images was provided by SideFX.


\section{Symplectic geometry of the space of curves}\label{sec:2}

Let $M\cong\R$ with a fixed orientation and let $\tau\in\Diff(M)$ denote a translation so that $M/\tau\cong S^1$. A smooth immersion $\gamma\colon M\to N$ into an oriented Riemannian manifold $N$ is called a {\em curve with monodromy} if 
\[
\tau^*\gamma=h \circ \gamma
\]
for an orientation preserving isometry $h$ of $N$, see Figure~\ref{fig:JustElastic}. The space of curves $\mathcal{M}=\mathcal{M}_{h}$ of a given monodromy $h$  is a smooth Fr{\' e}chet manifold, whose tangent spaces 
\[
T_{\gamma}\mathcal{M}=\Gamma_{\tau}(\gamma^*TN)
\]
are given by vector fields $\xi$ along $\gamma$ compatible with $\tau$, that is
\[
\tau^*\xi=h \xi\,. 
\]
We make no notational distinction between the action of an isometry on $N$ and its derived action on $TN$. The group of orientation preserving diffeomorphisms $\Diff_{\tau}(M)$ commuting with $\tau$ acts on $\mathcal{M}$ from the right by reparametrization.  We now assume that  $N$ is 3-dimensional with volume form $\det\in\Omega^3(N,\R)$. Then $\mathcal{M}$ carries a presymplectic structure \cite{marsden1983}, the {\em Marsden--Weinstein 2-form} $\sigma\in \Omega^2(\mathcal{M},\R)$, whose value at $\gamma\in\mM$ is given by
 \begin{equation}
\label{eq:Arnold}
\sigma_{\gamma}(\xi,\eta)=\int_{I}\det(d\gamma, \xi,\eta) \,.
\end{equation}
Here  $\xi,\eta\in T_{\gamma}\mathcal{M}$ and $I\subset M$ denotes a fundamental domain for $\tau$. Since the integrant is compatible with $\tau$, the integral is well-defined independent of choice of the fundamental domain. The Marsden--Weinstein 2-form $\sigma$ is degenerate exactly along the tangent spaces to the orbits of $\Diff_{\tau}(M)$. Thus, $\sigma$ descends to a symplectic structure on the space $\mM/\Diff_{\tau}(M)$ of unparametrized curves, which is singular at $\gamma\in\mM$ where $\Diff_{\tau}(M)$ does not act freely.
Since all our Hamiltonians will be geometric, and therefore  invariant under $\Diff_{\tau}(M)$, they also are defined on the space of unparametrized curves $\mM/\Diff_{\tau}(M)$.  Nevertheless, we always will carry out calculations on the smooth manifold $\mM$ and remind ourselves that results have to be viewed modulo the action of $\Diff_{\tau}(M)$.

\begin{figure}[ht]
  \centering
  \includegraphics[width=.8\columnwidth]{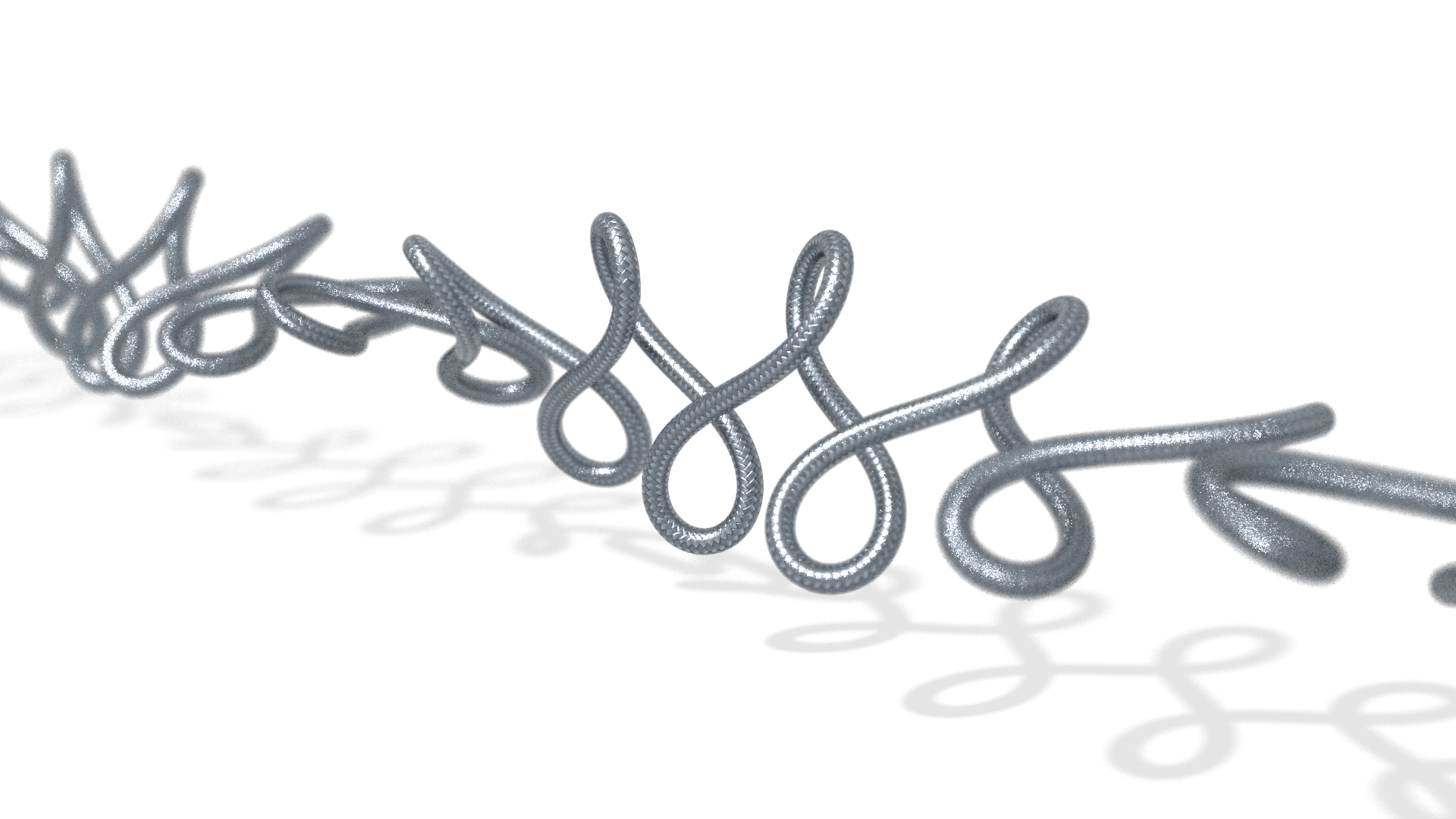}
  \caption{An elastic curve in $\R^3$ with monodromy.
  \label{fig:JustElastic}}
\end{figure}

In addition to its presymplectic structure, $\mM$ carries a Riemannian structure given by the variational $L^2$-metric
\begin{equation*}
\label{eq:L2metric}
\langle \xi,\eta\rangle_{\gamma}=\int_I(\xi,\eta)dx
\end{equation*}
for $\xi,\eta\in T_{\gamma}\mM$. Here $dx\in\Omega^1(M,\R)$ denotes the volume form of the induced metric $dx^2=(d\gamma,d\gamma)$ on $M$. We emphasize that the $1$-form $dx$ depends on the curve $\gamma$.  In classical terms $x$ is the arclength parameter for $\gamma$ on $M$. Derivatives, including covariant derivatives,  along a curve $\gamma\in\mathcal{M}$ with respect to the dual vector field $\tfrac{d}{dx}\in \Gamma(TM)$  to $dx$ will often be notated by $(\,)'$. For instance, $d\gamma=\gamma' dx$ and $T:=\gamma'$ is the unit tangent vector field along $\gamma$; if  $\xi$ is a section of a bundle with connection $\nabla$ over $M$, then $\nabla\xi=\xi'dx$.

To avoid overbearing notation, we will notationally not distinguish between $F$ and $F(\gamma)$ for maps $F$ defined on $\mM$. For example, $T=\gamma'$ can refer to the unit tangent vector field along $\gamma$ or the vector field 
$T\in\Gamma(T\mM)$, whose value at $\gamma$ is given by $T(\gamma)=\gamma'$.

The vector cross product, defined by the volume form $\det\in\Omega^3(N,\R)$, relates the Riemannian metric and the Marsden--Weinstein $2$-form  
on $\mM$  by
\begin{equation}
\label{eq:crossproduct}
\sigma (\xi,\eta)=\langle T\times \xi,\eta\rangle\,.
\end{equation}
Given a Hamiltonian  $E\colon \mM\to \R$, its {\em variational gradient} $G_E\in\Gamma(T\mM)$ is given by  
\[
\langle G_E,-\rangle=dE\,.
\]
The solutions of  the Euler--Lagrange equation $G_E=0$ are the critical points for the variational problem defined by $E$. We call a vector field $Y_E\in\Gamma(T\mM)$ on $\mM$ a {\em symplectic gradient} if 
\[
\sigma(Y_E,-)=dE\,.
\]
Note that due to the degeneracy of the Marsden--Weinstein 2-form,  $Y_E$ is defined only up to the tangent spaces to the orbits of $\Diff_{\tau}(M)$ given by 
$C^{\infty}(\mM,\R)T\subset T\mM$. Equation  \eqref{eq:crossproduct} relates the variational and symplectic gradients by
\begin{equation}
\label{eq:sgrad2grad}
G_{E}=T\times Y_{E}\,.
\end{equation}
A Hamiltonian $E$ is invariant under $\Diff_{\tau}(M)$ if and only if $G_{E}\in \Gamma_{\tau}(\perp_{\gamma}\!\!M)$ is a normal vector field along $\gamma\in\mM$. In this situation, relation \eqref{eq:sgrad2grad} can be reversed to 
\begin{equation}
\label{eq:grad2sgrad}
Y_E \equiv -T\times G_E\mod C^{\infty}(\mM,\R)T\,,
\end{equation}
and provides a symplectic gradient $Y_E$  for the invariant Hamiltonian $E$. 

There are two primary reasons why we work on the space  $\mM=\mM_{h}$ of curves with mono\-dromy, rather than just with closed curves
whose monodromy $h=\Id$ is trivial: first, a number of historically relevant examples, such a elastic curves, give rise to curves with monodromy;  second, curves with monodromy appear naturally in the {\em associated family} of curves, which provides the setup to calculate the commuting hierarchy of Hamiltonians on $\mM$ and, at the same time, gives a geometric  realization of the {\em spectral curve}. 

\section{Energy functionals on the space of curves}
Calculating the variational and symplectic gradients for a number of classical Hamiltonians $E\in C^{\infty}(\mM,\R)$,
we will observe a pattern leading to a recursion relation for an infinite hierarchy of flows.  Given a curve $\gamma\in\mM$, we choose a variation $\gamma_t\in\mM$ and denote by $(\,)\udot$ any kind of derivative with respect to $t$ at $t=0$. The variational gradient $G_E$ is then read off from 
\[
\dot{E}=dE_{\gamma}(\dot{\gamma})=\langle G_E,\dot{\gamma}\rangle\,.
\]
\subsection{Length functional}
\begin{figure}[ht]
  \centering
  \includegraphics[width=.4\columnwidth]{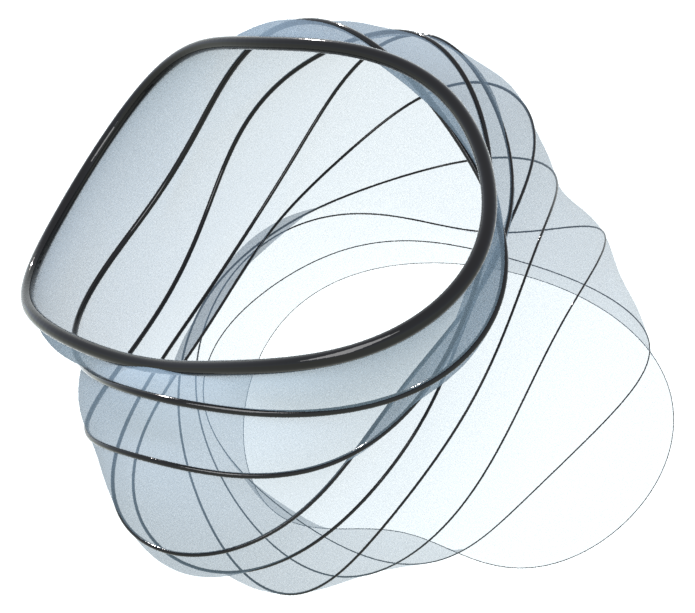}
  \caption{A closed curve in $\R^3$  evolving according to the vortex filament flow $Y_1$.
  \label{fig:smokering}}
\end{figure}

We begin with the most elementary Hamiltonian, the length functional $E_1=\int_Idx$, obtained by integrating the volume form $dx$.
\begin{The}\label{thm:lengthgrad}
The variational and symplectic gradients of the length functional $E_1$ are given by $G_1=-\gamma''$ and $Y_1\equiv \gamma'\times\gamma'' \mod C^{\infty}(\mM,\R)T$. The latter gives rise to the {\em vortex filament flow} \cite{daRios1906} on $\mM$, see Figure~\ref{fig:smokering}.
\end{The}
\begin{proof}
Since $d\gamma_t=T_t dx_t$, we have the formula
\begin{equation}\label{eq:dotgamma}
(d\gamma)\udot=\nabla\dot{\gamma}=\dot{\gamma}' dx=\dot{T}dx+Td\dot{x}
\end{equation}
which will be reused a number of times. Taking inner product with $T$, and recalling $(T,\dot{T})=0$  due to $(T,T)=1$, gives
\[
d\dot{x}=(\nabla\dot{\gamma},T)=d(\dot{\gamma},T)-(\dot{\gamma},T')dx\,.
\]
Integrating both sides over a fundamental domain $I\subset M$ and applying Stokes' Theorem yields
\[
\dot{E_1}=\int_{I}d\dot{x}=-\int_I (\dot{\gamma},T')dx= \langle -\gamma'',\dot{\gamma}\rangle\,.
\]
Hence $G_1=-\gamma''$ and from \eqref{eq:grad2sgrad} the symplectic gradient becomes $Y_1\equiv \gamma'\times \gamma''$.
\end{proof}
\subsection{Torsion functional}
The next functional to consider is the {\em total torsion} of a curve $\gamma\in\mM$. It compares the monodromy of the curve $\gamma$ with the holonomy of its normal bundle over a fundamental region $I\subset M$. The normal bundle $\perp_{\gamma}\!\!M$ is a unitary complex line bundle over $M$ with  
complex structure $T\times (-)$, and hermitian metric $(\,,\,)$ and unitary connection $\nabla$  induced by the corresponding Riemannian structures on $N$. If $P$ denotes the parallel transport in the normal bundle $\perp_{\gamma}\!\!M$ along a fundamental domain $I\subset M$ and $h$ the monodromy of the curve $\gamma$, then $h^{-1}P\in S^1$ is a unimodular complex number. Then the  {\em total torsion}  is given by 
\begin{equation*}\label{eq:torsion}
E_2=\alpha\,,\qquad h^{-1}P = \exp(i\alpha)\,.
\end{equation*}
Note that the  functional $E_2$ is only defined modulo $2\pi$. 
The following characterization of the total torsion will be useful for calculating its variational gradient:
\begin{Lem}\label{lem:torsion}
Let $\gamma\in\mM$ and $\nu\in T_{\gamma}\mM$ be a unit length normal vector field along $\gamma$. The total torsion can be expressed as
\[
\alpha=\int_I(\nabla \nu,T\times \nu)\,.
\]
\end{Lem}
\begin{proof}
The unit length section $\xi=\nu\exp(-i\omega)\in\Gamma(\perp_{\gamma}\!\!M)$ is parallel, that is $\nabla\xi=0$, if and only if $\nabla\xi=T\times\xi \,d\omega$. Since 
$(\nabla \nu, T\times \nu)=(\nabla \xi, T\times \xi)$, we obtain $\alpha=\omega(x_1)-\omega(x_0)$, where $\del I=\{x_0\}\cup \{x_1\}$  is the oriented boundary of the fundamental domain $I\subset M$. On the other hand
\begin{align*}
h^{-1}P&=(\xi(x_1),h\xi(x_0))\\
&=(\nu(x_1)\exp(-i\omega(x_1)), \nu(x_1)exp(-i\omega(x_0)))\\
&=\exp(i(\omega(x_1)-\omega(x_0))=\exp(i\alpha))\,.
\end{align*}
\end{proof}
Even though the total torsion functional $E_2$ is defined modulo $2\pi$, it has a well defined variational gradient.
\begin{The}\label{thm:torsiongrad}
The variational gradient of the total torsion $E_2$  is given by
\[
G_2=-(\gamma'\times\gamma''' +*R(T)T)
\]
where $*$ denotes the Hodge star operator on the 3-dimensional oriented Riemannian manifold $N$ and $R\in\Omega^2(N,{\bf so}(TN))$ its Riemannian curvature 2-form. In particular, if $N$ is a 3-dimensional space form, the variational and symplectic gradients of the total torsion have the simple expression
\[
G_2=-\gamma'\times\gamma'''\quad\text{and}\quad Y_2\equiv -\gamma'''\mod C^{\infty}(\mM,\R)T \,.
\]
The flow generated by $Y_2$ is called the {\em helicity filament flow}  \cite{holm2004} on $\mM$.
\end{The}
\begin{proof} 
Let $\gamma_t\in\mM$  be a variation of $\gamma$. Using the characterization of Lemma~\ref{lem:torsion}, we calculate
\[
(\nabla \nu,T\times \nu)\udot= ((\nabla \nu)\udot, T\times \nu)+(\nabla \nu, \dot{T}\times \nu)+(\nabla \nu, T\times \dot{\nu})\,.
\]
The last term $(\nabla \nu, T\times \dot{\nu})=\det(T,\dot{\nu},\nabla \nu)=0$, since all entries are perpendicular to $\nu$ and $\nu^{\perp}$ is 2-dimensional. In the first term, we commute $(\,)\udot$ with $\nabla$ and accrue a curvature term
\[
(\nabla \nu)\udot=\nabla\dot{\nu}+R(\dot{\gamma},d\gamma)\nu\,.
\]
Therefore, using the symmetries of the Riemannian curvature $R$, we arrive at 
\[
(\nabla \nu,T\times \nu)\udot= (\nabla\dot{ \nu}, T\times \nu)+(\nabla \nu, \dot{T}\times \nu)-(R(\nu,T\times \nu)d\gamma,\dot{\gamma})\,.
\]
The curvature term can easily be rewritten as
 \[
R(\nu,T\times \nu)d\gamma =*R(T)T dx
\]
using the Hodge star  $*$ on $N$ and the relation $d\gamma=\gamma'\,dx$.  It remains to unravel the first two terms of the right hand side in the expression
above:
\[
 (\nabla\dot{ \nu}, T\times \nu)+(\nabla \nu, \dot{T}\times \nu)=d(\dot{\nu},T\times \nu)-(\dot{\nu},\nabla T\times \nu)+(\nabla \nu,\dot{T}\times \nu)
 \]
 where we omitted the term $(\dot{\nu},T\times \nabla \nu)$, which again vanishes due to all entries being perpendicular to $\nu$. Since the terms $\nabla T\times \nu$ and $\dot{T}\times \nu$  are tangential, we can rewrite 
 \[
 (\dot{\nu},\nabla T\times \nu)=(\dot{\nu},T)(T,\nabla T\times \nu)=(\nu,\dot{T})(\nabla T, T\times \nu)
 \]
 and likewise for $(\nabla \nu,\dot{T}\times \nu)$, where we used properties of the cross product.  Putting all this together and calculating modulo exact $1$-forms yields 
 \begin{align*}
 (\nabla\dot{ \nu}, T\times \nu)+(\nabla \nu, \dot{T}\times \nu)&=(\dot{T}, -\nu(\nabla T, T\times \nu)+T\times \nu (\nu,\nabla T))\\
 &=(\dot{T}, T\times \nabla T)\,.
 \end{align*}
 Next we replace $\dot{T}$ by \eqref{eq:dotgamma} and notice that $T\times \nabla T$ is normal, so that 
 \begin{align*}
 (\dot{T}, T\times \nabla T)&=(\dot{\gamma}', T\times \nabla T)\\
 &=d(\dot{\gamma}, T\times T')-(\dot{\gamma},T\times T'')dx\,.
 \end{align*}
Therefore, applying Stokes' Theorem, we have shown that 
\begin{align*}
\dot{E_2}&=\int_I  (\nabla \nu,T\times \nu)\udot=\int_I (\dot{\gamma},-T\times T''-*R(T)T)dx \\
&=\langle -T\times T'' -*R(T)T , \dot{\gamma}\rangle\,,
 \end{align*}
 which gives the stated result for the variational gradient $G_2$ after putting $T=\gamma'$. The vanishing of $*R(T)T$ for a space of constant curvature follows from the special form of the curvature tensor $R$, since $T$ is perpendicular to $\nu$ and $T\times \nu$. The expression for the symplectic gradient $Y_2$ comes from $\eqref{eq:grad2sgrad}$.
\end{proof}
\subsection{Elastic energy}
The next functional to be discussed is the {\em bending} or {\em elastic energy}
\begin{equation*}\label{eq:bending}
E_3=\tfrac{1}{2}\int_I |\gamma''|^2dx\,,
\end{equation*}
the total squared length of the curvature $\gamma''\in\Gamma_{\tau}(\perp_{\gamma}\!\!M)$ of a curve $\gamma\in\mM$ over a fundamental domain $I\subset M$. This functional has been pivotal in the development of many areas of mathematics, including variational calculus, geometric analysis, and integrable systems. In comparison, the previously discussed  total torsion appears to be less familiar, presumably because it is trivial for planar curves which, historically, were the first to be studied.
\begin{The}\label{thm:elasticgrad}
The variational gradient of the elastic energy $E_3$ is given by
\[
G_3=(\gamma'''+\tfrac{3}{2}|\gamma''|^2\gamma')'-R(\gamma',\gamma'')\gamma'\,.
\]
with $R\in\Omega^2(N,{\bf so}(TN))$ the Riemannian curvature. In case $N$ has constant curvature $K$, the variational gradient simplifies to 
\[
G_3=(\gamma'''+(\tfrac{3}{2}|\gamma''|^2-K)\gamma')'\,.
\]
\end{The}
\begin{proof}
Let $\gamma_t\in\mM$  be a variation of $\gamma$. Replacing $\gamma'=T$ in the integrant of $E_3$, we obtain
\[
\tfrac{1}{2}(|T'|^2dx)\udot=((T')\udot,T')dx+\tfrac{1}{2}|T'|^2 d\dot{x}
\]
where $d\dot{x}=(\dot{\gamma}',T)dx$ from \eqref{eq:dotgamma}. To evaluate the first term on the right hand side, we take the $t$-derivative of $\nabla T=T' dx$ to calculate $(T')\udot$. This gives 
\[
(\nabla T)\udot=(T')\udot dx+ T'd\dot{x}
\]
and, commuting $\nabla$ and $t$-derivatives,
\[
(\nabla T)\udot=\nabla\dot{T} + R(\dot{\gamma},d\gamma)T
\]
 accruing a curvature term. Together with our formula for $d\dot{x}$, this provides the expression
\[
(T')\udot=\dot{T}'+R(\dot{\gamma},T)T-T'(\dot{\gamma}',T)\,.
\]
Therefore,
\[
\tfrac{1}{2}(|T'|^2dx)\udot=(\nabla \dot{T},T')-\tfrac{1}{2}|T'|^2(\nabla\dot{\gamma},T)-(R(T,\nabla T)T,\dot{\gamma})\,,
\]
where we used the symmetries of the Riemannian  curvature $R$ and $\nabla T=T'dx$.  It remains to calculate modulo exact $1$-forms
\begin{align*}
(\nabla \dot{T},T')&=d(\dot{T},T')-(\dot{T}dx, T'')\equiv-(\nabla\dot{\gamma},T'')+(\nabla\dot{\gamma},T)(T,T'')\\
&=-d(\dot{\gamma},T'')+(\dot{\gamma},\nabla T'')-(\nabla\dot{\gamma},T)|T'|^2\,,
\end{align*}
where we used \eqref{eq:dotgamma} and the unit length of $T$. Inserting this last expression yields
\begin{align*}
\tfrac{1}{2}(|T'|^2 dx)\udot & =(\dot{\gamma},\nabla T'')-\tfrac{3}{2} |T'|^2(\nabla\dot{\gamma},T)-(R(T,\nabla T)T,\dot{\gamma})\\
&\equiv (\nabla ( T''+\tfrac{3}{2} |T'|^2 T)- R(T,\nabla T)T, \dot{\gamma})\,.
\end{align*}
Integrating over a fundamental domain $I\subset M$ and using Stokes' Theorem, we obtain
\begin{align*}
\dot{E_3}&=\int_I ((T''+\tfrac{3}{2} |T'|^2 T)'- R(T,T')T, \dot{\gamma})dx\\
&=\langle(T''+\tfrac{3}{2} |T'|^2 T)'- R(T,T')T,\dot{\gamma}\rangle\,,
\end{align*}
which gives the stated result for the variational gradient $G_3$. The expression of $G_3$ for a space form $N$ derives from the
special form of the curvature tensor $R$ in that case.
\end{proof}
\subsection{Flux functional}
There is  another, somewhat less known, Hamiltonian which will play a role in isoperimetric characterizations of critical points of the functionals $E_k$. If $V\in\Gamma(TN)$ is a volume preserving vector field on $N$, for example a Killing field, then $i_V\det\in\Omega^2(N,\R)$ is a closed $2$-form which has a primitive $\alpha_V\in\Omega^1(N,\R)$, provided the space form $N$ has no second cohomology. We additionally assume  that we can choose the primitive $\alpha_V$ invariant under the monodromy $h$, that is
$h^*\alpha_V=\alpha_V$. In particular, this will imply that the vector field $V$ is $h$-related to itself. Then we can integrate the primitive  
$\alpha_{V}$ along $\gamma$ over a fundamental domain $I\subset M$ to obtain the {\em flux functional} 
\begin{equation*}
\label{eq:flux}
E_{V}=\int_I \gamma^*\alpha_{V}\,.
\end{equation*} 
It is worth noting that on closed curves the flux functional is indeed the flux of the vector field $V$ through any disk spanned by $\gamma$. The primitive $\alpha_{V}$ of $i_{V}\det$ is defined only up to  a closed $1$-form and the value of $E_{V}$ depends on this choice. When calculating the variational gradient $G_V$ of $E_V$ only homotopic curves (of the same monodromy) play a role, along which the value of $E_V$ is independent of the choice of primitive by Stokes' Theorem.  
\begin{The}\label{thm:fluxgrad}
The variational and symplectic gradients of the flux functional $E_V$ are given by
\[
G_V=\gamma'\times (V\circ\gamma)\quad\text{and}\quad Y_V\equiv V\circ \gamma \mod C^{\infty}(\mM,\R)T\,.
\]
The Hamiltonian flow of $E_V$ is given by moving  an initial curve $\gamma$ along the flow of the vector field $V$ on $N$. 
\end{The}
\begin{proof}
If $\gamma_t\in\mM$ is a variation of $\gamma$, then 
\[
(\gamma_t^*\alpha_V)\udot=d i_{\dot{\gamma}}\alpha_V + i_{\tfrac{d}{dt}}(\gamma^* d\alpha_V)=d i_{\dot{\gamma}}\alpha_V +\det(V\circ\gamma,\dot{\gamma},d\gamma)\,,
\]
where we used the Cartan formula for the Lie derivative and $i_V\det=d\alpha_V$. Since 
\[
\det(V\circ\gamma,\dot{\gamma},d\gamma)=(d\gamma\times V\circ\gamma,\dot{\gamma})
\]
Stokes' Theorem gives us 
\[
(G_V)\udot=\int_I(\gamma'\times V\circ\gamma,\dot{\gamma})dx=\langle \gamma'\times V\circ\gamma, \dot{\gamma}\rangle\,,
\]
and we read off the formula for the variational gradient $G_V$. Applying \eqref{eq:grad2sgrad}, we obtain the symplectic gradient $Y_V=V\circ\gamma$. 
 \end{proof}
 As an instructive example, we look at the special case $N=\R^3$, where there are two obvious choices for $V$. For unit length $v\in\R^3$, we have the infinitesimal translation $V(p)=v$ and infinitesimal rotation $V(p)=v\times p$ giving rise to two more Hamiltonians $E_V$.
 The monodromy $h$ for the space of curves $\mM$ in $\R^3$ is an orientation preserving Euclidean motion $h=Ap+a$.
 The condition that $V$ is $h$-related to itself means that $v$ is an eigenvector of the rotational monodromy $A$.
 To simplify notation, we choose the origin of $\R^3$ to lie on the axis $\R v$. We can easily find the $h$-invariant primitives $\alpha_V\in\Omega^1(\R^3,\R)$ of $i_V\det$: for an infinitesimal translation $\alpha_V=\det(v,p,dp)$
and for an infinitesimal rotation $\alpha_V=\tfrac{1}{2}|p-(p,v)v|^2(v,dp)$.
 \begin{The}
 Let $N=\R^3$ and let $V(p)=v$ be an infinitesimal translation with unit length $v\in\R^3$ an eigenvector of the rotational part of the monodromy $h=Ap+a$. Then the flux functional becomes
 \[
 E_{-1}=\tfrac{1}{2}\int_{I}(\gamma\times d\gamma, v)\,.
 \]
 The vector $\tfrac{1}{2}\int_{I}\gamma\times d\gamma$, in reference to Kepler, is called the {\em area vector} of $\gamma$ and we call $E_{-1}$ the {\em area functional}. Indeed, for a closed curve $E_{-1}$ measures the signed enclosed area 
of the orthogonal projection of the curve $\gamma\in\mM$ onto the plane $v^{\perp}$. 
 
The variational and symplectic gradients of $E_{-1}$ are given by
 \[
G_{-1}=\gamma'\times v\quad\text{and}\quad Y_{-1}\equiv v \mod C^{\infty}(\mM,\R)T\,.
\]
\end{The}
The proof follows immediately from Theorem~\ref{thm:fluxgrad}.  
\begin{The}
Let $N=\R^3$ and let $V(p)=v\times p$ be an infinitesimal rotation with unit length eigenvector $v\in\R^3$  of the rotational part of the monodromy $h=Ap+a$. Then the flux functional becomes
 \[
 E_{-2}=\tfrac{1}{2}\int_{I}(|\gamma-(\gamma,v)v|^2 d\gamma, v)\,.
 \]
The vector $\tfrac{1}{2}\int_{I}(|\gamma-(\gamma,v)v|^2 d\gamma$ is called the {\em volume vector} of $\gamma$ and we call $E_{-2}$ the {\em volume functional}. Indeed, for a closed curve $E_{-2}$ measures the volume  
of the solid torus of revolution generated by rotating the curve $\gamma\in\mM$  around the axis $\R v$. 
 
The variational and symplectic gradients of $E_{-2}$ are given by
 \[
G_{-2}=\gamma'\times (v\times \gamma)\quad\text{and}\quad Y_{-2}\equiv v\times\gamma\mod C^{\infty}(\mM,\R)T\,.
\]
\end{The}
Again, the proof follows from Theorem~\ref{thm:fluxgrad} and an elementary calculation involving surfaces of revolution.

It will be convenient to add the {\em vacuum\,}, the trivial Hamiltonian $E_0=0$ with trivial variational gradient $G_0=0$ and symplectic gradient $Y_0=\gamma'$.  The Hamiltonian flow generated by $Y_0$ on $\mathcal{M}$  is reparametrization and thus descends to the  trivial flow on the quotient $\mathcal{M}/\Diff_{\tau}(M)$.
\begin{figure}[ht]
\centering
\includegraphics[width=.7\columnwidth]{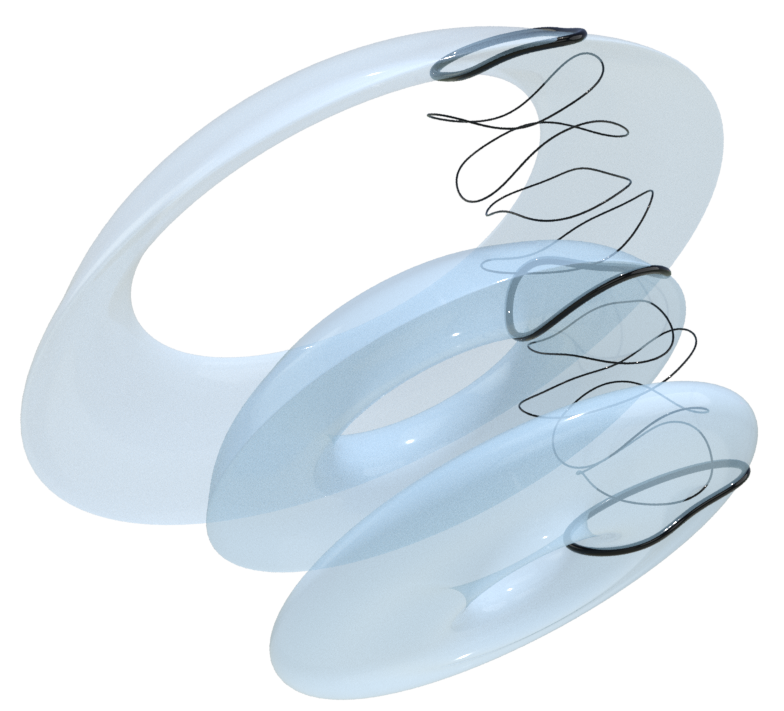}
\caption{A closed curve $\gamma$ moving according to vortex filament flow and the full tori obtained from $\gamma$ by rotation a disk bounded by $\gamma$ around a fixed axis. All these full tori have the same volume.
\label{fig:volumefunctional}}
\end{figure}

\subsection{Hierarchy of flows}
One of the reasons to list those functionals, besides their historical significance, is the curious pattern 
\begin{equation}
\label{eq:prerecursion}
G_{k+1}=-Y_k'
\end{equation}
their gradients seem to exhibit for curves in Euclidean 3-space $\R^3$, at least for $-2\leq k\leq 1$.  The pattern breaks at $k=2$ since $G_3$, as a variational gradient of the geometrically given energy functional $E_3$, is normal to $\gamma$ but $Y_2'=-\gamma^{\textrm{(iv)}}$ is not. This suggests that our initial choice of $Y_2$ should be augmented by a tangential component in order to have $Y_2'$ normal.  
Now for any vector field $Y\in\Gamma(T\mM)$, we have $d\dot{x}=(\nabla\dot{\gamma},T)$ along its flow $\dot{\gamma}=Y(\gamma)$. Therefore, the condition that $Y'$ is normal to $\gamma$ is equivalent to the fact that all curves $\gamma_t$ of the flow 
have the same induced volume form $dx_t=dx$, that is, the same parametrization.
In the case of $Y_2$ this leads to the tangential correction term $-\tfrac{3}{2}|\gamma''|^2\gamma'$. 
Therefore, by combining \eqref{eq:sgrad2grad} with the recursion equation \eqref{eq:prerecursion}, we arrive at the recursion 
\begin{equation}
\label{eq:recursion}
Y_{k}'+T\times Y_{k+1}=0\,,\quad Y_0=T
\end{equation}
defining an infinite  hierarchy of vector fields $Y_k\in\Gamma(T\mM)$.
\begin{The}[\cite{langer1999,yasui1998}]\label{thm:recursionsolution}
The recursion \eqref{eq:recursion} has the explicit solution
\[
Y_{k+1}=T\times Y_{k}'-\frac{1}{2}\sum_{i=1}^{k} (Y_i,Y_{k-i+1})T\,,\quad k\geq 0\,.
\]
In particular, $Y_k$ is a homogeneous differential polynomial of degree $k$ in $T$, and thus a differential polynomial of degree $k+1$ in $\gamma$. Moreover, if $\gamma$ is contained in a plane $E\subset \R^3$, then $Y_{2k}(\gamma)\colon M\to E$ is planar, whereas $Y_{2k+1}(\gamma)\colon M\to E^{\perp}$ is normal. In other words, the sub-hierarchy of even flows preserves planar curves.
\end{The}
\begin{proof}
Our prescription for the flows $Y_k$ demands a tangential correction to \eqref{eq:grad2sgrad} of the form
\[
Y_k=-T\times G_k+f_k T=T\times Y_{k-1}'+f_k T
\]
in order for \eqref{eq:prerecursion} to hold. Since $f_k=(Y_k,T)$ the first few $f_k$ can easily be read off: $f_0=1$, $f_1=0$, and $f_2=-\tfrac{1}{2}|T'|^2$. To calculate a general expression, we first note that
\[
(G_k,Y_l)=-(G_k,T\times G_l)=(T\times G_k,G_l)=-(Y_k,G_l)
\]
where we used \eqref{eq:grad2sgrad}. On the other hand, using \eqref{eq:prerecursion}, we obtain
\begin{align*}
(G_k,Y_l)&=(-Y_{k-1}',Y_l)=-(Y_{k-1},Y_l)'+(Y_{k-1},Y_l')=-(Y_{k-1},Y_l)'-(Y_{k-1},G_{l+1})\\
&=-(Y_{k-1},Y_l)'+(G_{k-1}, Y_{l+1})\,.
\end{align*}
Recalling $G_0=0$, this relation telescopes to
\begin{equation}\label{eq:telescoped}
(G_k,Y_l)=-\sum_{i=1}^{k}(Y_{k-i},Y_{l+i-1})'\,,
\end{equation}
which, together with \eqref{eq:prerecursion}, \eqref{eq:sgrad2grad} and $T=Y_0$, implies
\begin{align*}
f_k'&=(Y_k,T)'=(Y_k',T)+(Y_k,T')=(Y_k,Y_0')=-(Y_k,G_1)=(G_k,Y_1)\\
&=-\sum_{i=1}^{k-1}(Y_{k-i},Y_{i})' - (Y_0,Y_k)'=-\sum_{i=1}^{k-1}(Y_{k-i},Y_{i})'- f_k'\,.
\end{align*}
This shows that (up to a constant of integration) $f_k=-\tfrac{1}{2}\sum_{i=1}^{k-1}(Y_{k-i},Y_{i})$. 
The remaining statements of the theorem follow from simple induction arguments.
\end{proof}
At this stage it is worth to discuss some  of the immediate consequences of the last theorem.
\begin{Cor}\label{cor:precommute}
If  $Y_k$ are symplectic gradients for Hamiltonians $E_k\in C^{\infty}(\mM,\R)$, then $E_k$  is constant along every flow $Y_l$ for $k,l\geq=-2$. 
\end{Cor}
\begin{proof}
The relation \eqref{eq:telescoped} implies that 
\[
dE_k(Y_l)=\int_I (G_k,Y_l)dx =0
\]
at least for $k,l\geq 0$. One can easily derive an analogous relation to  \eqref{eq:telescoped} for $-2\leq k,l\leq 0$, which proves the claim.
\end{proof}
As an example, we apply this result to the vortex filament flow $\dot{\gamma}=\gamma'\times\gamma''$  of a closed initial curve $\gamma_0$. Corollary~\ref{cor:precommute} then implies that the enclosed volumes $E_{-2}(\gamma_t)=E_{-2}(\gamma_0)$ of the solid tori of revolution
around the axis $\R v$  generated by $\gamma_t$  are equal for all times $t$ in the existence interval of the flow, see Figure~\ref{fig:volumefunctional}. Similarly, the projected areas $E_{-1}(\gamma_t)$ of the curves $\gamma_t$ are all equal.

Corollary~\ref{cor:precommute} also provides evidence as to why the flows $Y_k$ should commute: assuming $Y_k$ are symplectic flows to geometric, that is, parametrization invariant, Hamiltonians $E_k$, they descend to the quotient $\mM/\Diff_{\tau}(M)$. The latter  is symplectic, at least away from its singularities. By standard symplectic techniques Corollary~\ref{cor:precommute} then implies that the flows on the (non-singular part of the) quotient commute. This would at least show that the flows $Y_k$ on $\mM$ commute modulo $T$. More is true though, as we shall discuss in Section~\ref{sec:Laxflow}, where we relate the flows $Y_k$ to Lax flows on a loop algebra to show that they genuinely commute.

\begin{table}[ht]
  \centering
  \caption{\label{tab:hierarchy} The variational and symplectic gradients of the listed functionals for curves in $\R^3$.}
  \tabulinesep=1mm
  \begin{tabu}{|p{0.35\textwidth}|p{0.25\textwidth}|p{0.285\textwidth}|}
	\hline
    Functional & Variational gradient & Symplectic gradient\\
	\hline
	$E_{-2}=\tfrac{1}{2}\int_{I}(|\gamma-(\gamma,v)v|^2 d\gamma, v)$
	&
	$G_{-2}=\gamma'\times (v\times \gamma)$ & 
	$Y_{-2}=v\times\gamma$ \\
	\hline
	$E_{-1}=\tfrac{1}{2}\int_{I}(\gamma\times d\gamma, v)$ &
	$G_{-1}=\gamma'\times v$ &
	$ Y_{-1}=v$ \\
	\hline
	$E_0 = 0$ & $G_0 = 0$ & $Y_0 = \gamma'$\\
	\hline
	$E_1 = \int_{I} dx$ &
	$G_1 = -\gamma''$ &
	$Y_1 = \gamma'\times\gamma''$\\
	\hline
	$E_2 = \alpha$&
	$G_2 = -\gamma'\times\gamma'''$&
	$Y_2 = -\gamma''' - \tfrac{3}{2}|\gamma''|^2\gamma'$\\
	\hline
	$E_3 = \tfrac{1}{2}\int_{I}|\gamma''|^2dx$&
	$G_3 = (\gamma'''+\tfrac{3}{2}|\gamma''|^2\gamma')'$&
	\makecell[lt]{
	$Y_3 = -\gamma'\times \gamma''''$\\
	$\phantom{Y_3={}}-\tfrac{3}{2}|\gamma''|^2\gamma'\times\gamma''$\\
	$\phantom{Y_3={}}+\det(\gamma',\gamma'',\gamma''')\gamma'$
	}\vspace{-2mm}\\
	\hline
  \end{tabu}
\end{table}
 
\subsection{Constrained critical curves in $\R^3$} 
We continue to denote by $E_k\colon\mM\to\R$ the putative Hamiltonians for the flows $Y_k\in\Gamma(T\mM)$. 
\begin{Def}
A curve $\gamma\in\mathcal{M}$ is critical for the functional $E_{k+1}$ under the constraints
$E_1\dots, E_{k}$, if the constrained Euler--Lagrange equation
\[
G_{k+1}=\sum_{i=1}^{k}c_i G_i
\]
holds for some constants, the Lagrange multipliers,  $c_i\in\R$. 
\end{Def}
The recursion relation \eqref{eq:prerecursion} implies that the condition for an $E_{k+1}$ critical constrained curve is equivalent to 
\begin{equation}\label{eq:criticalY}
Y_{k}=\sum_{i=0}^{k-1}c_i Y_i+ v
\end{equation}
for some $c_i\in\R$, where $v\in\R^3$ necessarily is an eigenvector of the rotational part of the monodromy of $\gamma$. In terms of flows this says that $Y_k$ evolves an initial curve $\gamma$, modulo lower order flows, by a pure translation in the direction of the monodromy axis. Since $G_k=T\times Y_k$ by \eqref{eq:sgrad2grad}, the above is equivalent to 
\[
G_{k}=\sum_{i=0}^{k-1}c_i G_i+ T\times v = \sum_{i=-1}^{k-1}c_i G_i\,.
\]
This last says that $\gamma$ is $E_k$ critical under the  constraints $E_{-1},\dots, E_{k-1}$, where we remember $E_0=0$. Applying this construction once more gives the equivalent condition
\[
G_{k-1}=\sum_{i=-2}^{k-2}c_i G_i\,,
\]
that is to say, $\gamma$ is $E_{k-1}$ critical under the constraints $E_2,\dots, E_{k-2}$.  
\begin{The}
The following statements for a curve $\gamma\in\mM$  are equivalent:
\begin{enumerate}
\item
$\gamma$ is $E_{k+1}$ critical under constraints $E_1,\dots, E_k$.
\item
$\gamma$ is $E_{k}$ critical under constraints $E_{-1},\dots, E_{k-1}$.
\item
$\gamma$ is $E_{k-1}$ critical under constraints $E_{-2},\dots, E_{k-2}$.
\item
The flow $\dot{\gamma}=Y_k$ evolves, modulo a linear combination of lower order flows, by a pure translation in the direction of the axis of the mondoromy of the initial curve $\gamma$.
\item
The flow $\dot{\gamma}=Y_{k-1}$ evolves, modulo a linear combination of lower order flows, by a Euclidean motion with axis along the monodromy of the initial curve $\gamma$.
\end{enumerate}
\end{The}
From this observation, we obtain a striking isoperimetric characterization of {\em Euler elastica}, that is, curves critical for the elastic energy $E_3$ constrained by the length $E_1$ and total torsion $E_2$. 
\begin{Cor}
The following conditions on a curve $\gamma\in\mM$ are equivalent:
\begin{enumerate}
\item
$\gamma$ is an Euler elastic curve.
\item
$\gamma$ is critical for total torsion with length and area $E_{-1}$ constraints.
\item
$\gamma$ is critical for length with area and volume $E_{-2}$ constraints.
\item
The helicity filament flow $\dot{\gamma}=Y_2$ evolves, modulo the vortex filament flow and reparametrization, by a pure translation in the direction of the axis of the mondoromy of the initial curve $\gamma$.
\item
The vortex filament flow $\dot{\gamma}=Y_1$ evolves, modulo reparametrization, by a Euclidean motion with axis along the monodromy of the initial curve $\gamma$.
 \end{enumerate}
\end{Cor}

\section{Lax flows on loop algebras}\label{sec:Laxflow}
We have constructed a hierarchy of vector fields $Y_k$ on $\mM$ which are symplectic for $-2\leq k\leq 3$ with respect to explicitly given Hamiltonians $E_{-2},\dots,E_3$, see Table~\ref{tab:hierarchy}. In order for our hierarchy $Y_k$ to qualify as an  ``infinite dimensional integrable system'', the least we need to verify is
\begin{enumerate}
\item
The vector fields $Y_k$ commute, that is $[Y_i,Y_j]_{\mM}=0$, for which Corollary~\ref{cor:precommute} provides evidence. 
\item
There is a hierarchy of Hamiltonians $E_k\in C^{\infty}(\mM,\R)$ for which $Y_k$ are symplectic. 
\end{enumerate}
To this end it will be helpful to rewrite the recursion \eqref{eq:recursion} in terms of the generating function 
\begin{equation*}\label{eq:generating}
Y=\sum_{k\geq 0}Y_k\lambda^{-k}
\end{equation*}
of the hierarchy $Y_k\in\Gamma(T\mathcal{M})$ as 
\begin{equation}\label{eq:LaxY}
Y'+\lambda T\times Y=0\,.
\end{equation}
This last is a Lax pair equation on the loop Lie algebra 
\begin{equation*}\label{eq:loopalgebra}
\Lambda\R^3=\R^3((\lambda^{-1}))
\end{equation*}
of formal Laurent series with poles at $\lambda=\infty$. The Lie algebra structure on $\Lambda\R^3$ is the pointwise structure induced from the Lie algebra $(\R^3,\times)$. 
From the construction of $Y_k$ in Theorem~\ref{thm:recursionsolution}, we see that 
\begin{equation}\label{eq:Yunit}
(Y,Y)=\sum_{k,l\geq 0}(Y_k,Y_l)\lambda^{-k-l}=1\,.
\end{equation}
One way  to address the commutativity of the flows $Y_k$ is to relate them to commuting vector fields on $\Lambda\R^3$ obtained by  standard methods, such as Adler--Kostant--Symes or, more generally, R-matrix theory.  

Without digressing too far, we summarize the gist of these constructions \cite{ferus1992}:  given a Lie algebra $\frak{g}$, a vector field $X\colon\frak{g}\to\frak{g}$ is $\ad$-invariant if 
\begin{equation}\label{eq:adinv}
dX_{\xi}[\xi,\eta]=[X(\xi),\eta]\,,\quad \xi,\eta\in\frak{g}\,.
\end{equation}
Standard examples of $\ad$-invariant vector fields are gradients (with respect to an invariant inner product)  of $\Ad$-invariant functions if $\frak{g}$ is the Lie algebra of a Lie group $G$. If $X,\tilde{X}$ are $\ad$-invariant vector fields, then \eqref{eq:adinv} implies
\[
 [\xi,X(\xi)]=0\quad\text{and}\quad[X(\xi),\tilde{X}(\xi)]=0
 \]
 for $\xi\in\frak{g}$. These are just a restatement of the fact that gradients of $\Ad$-invariant functions commute with respect to the Lie--Poisson structure on $\frak{g}$. One further ingredient is the so-called R-matrix, that is, an endomorphism $R\in\End(\frak{g})$ satisfying the Yang--Baxter equation
 \[
R( [R\xi,\eta]+[\xi,R\eta])-[R\xi,R\eta]=c[\xi,\eta]
\]
for some constant $c$. This relation is a necessary condition for $[\xi,\eta]_R= [R\xi,\eta]+[\xi,R\eta]$ to satisfy the Jacobi identity. Thus, $[\,,\,]_R$ defines a second Lie algebra structure  on $\frak{g}$. A typical example of an R-matrix (relating to the Adler--Kostant--Symes method)  arises from a decomposition 
$\frak{g}=\frak{g}_{+}\oplus\frak{g}_{-}$ into Lie subalgebras: $R=\tfrac{1}{2}(\pi_{+}-\pi_{-})$ with $\pi_{\pm}$ the projections along the decomposition. Whereas the Hamiltonian dynamics of $\Ad$-invariant functions with respect to the original Lie-Poisson structure on $\frak{g}$ is trivial, the  introduction of an R-matrix renders this dynamics non-trivial for the new Lie--Poisson structure while keeping the original commutativity.  Given an $\ad$-invariant vector field $X$ on $\frak{g}$ and a constant $\mu$, then
\[
V(\xi)=[\xi,(R+\mu \Id)X(\xi)]
\]
is the corresponding  ``Hamiltonian'' vector field for the Lie--Poisson structure corresponding to $[\,,\,]_R$.  We collect the basic integrability properties of the vector fields $V$, which are verified by elementary calculations using the properties of $\ad$-invariant vector fields.
\begin{Lem}\label{lem:fundamental}
Let $X,\tilde{X}$ be $\ad$-invariant vector fields on the Lie algebra $\frak{g}$.
\begin{enumerate}
\item
The corresponding Hamiltonian vector fields $V,\tilde{V}$ on $\frak{g}$ commute, that is, their vector field Lie bracket $[V,\tilde{V}]_{C^{\infty}}=0$.
\item
Let $\xi\colon D\to \frak{g}$ be a solution to 
\begin{equation}\label{eq:Lax}
d\xi=V(\xi)dt+\tilde{V}(\xi)d\tilde{t}
\end{equation}
over a 2-dimensional domain $D$. Then the $\frak{g}$-valued $1$-form 
\[
\beta=(R+\mu \Id)X(\xi) dt +(R+\tilde{\mu} \Id)\tilde{X}(\xi) d\tilde{t}
\]
on $D$ satisfies the Maurer-Cartan equation $d\beta+\tfrac{1}{2}[\beta\wedge\beta]=0$. 
\end{enumerate}
If $G$ is a Lie group for $\frak{g}$, the ODE
\[
dF=F\beta
\]
has a solution $F\colon D\to G$ by (ii) and $\xi=\Ad F^{-1}\eta$ solves \eqref{eq:Lax} for any initial condition $\eta\in\frak{g}$. Thus, $\Ad$-invariant functions are constant along solutions to \eqref{eq:Lax}, which is referred to as ``isospectrality''  of the flow. Specifically, if an $\Ad$-invariant inner product  on $\frak{g}$ is positive definite and the flow \eqref{eq:Lax} is contained in a finite dimensional subspace of $\frak{g}$, then the flow is complete. 
\end{Lem}
With this at hand, the recursion equation \eqref{eq:LaxY} for the generating loop $Y=\sum_{k\geq 0}Y_k\lambda^{-k}$  can be seen as the first in a hierarchy of commuting flows on the loop algebra $\Lambda\R^3$. We have the splitting
\begin{equation*}\label{eq:splitting}
\Lambda\R^3=\Lambda^{+}\R^3\oplus\Lambda^{-}\R^3
\end{equation*}
into the Lie subalgebras $\Lambda^{+}\R^3=\R^3[\lambda]_{0}$ of polynomials without constant term and power series 
$\Lambda^{-}\R^3 = \R^3[[\lambda^{-1}]]$ in $\lambda^{-1}$ with R-matrix $R=\tfrac{1}{2}(\pi_{+}-\pi_{-})$. 
For $k\in\Z$ the vector fields 
\[
X_k(\xi)=\lambda^{k+1}\xi
\]
on $\Lambda\R^3$ are $\ad$-invariant with corresponding commuting vector fields
\[
V_k(\xi)=\xi \times (R+\tfrac{1}{2}\Id)X_k(\xi) =\xi\times (\lambda^{k+1}\xi)_{+}\,
\] 
as shown in  Lemma~\ref{lem:fundamental}. For $\xi\in\Lambda^{-}\R^3$ we have
\[
\xi \times V_{k}(\xi)=\xi\times (\lambda^{k+1}\xi)_{+}=- \xi\times (\lambda^{k+1}\xi)_{-}\,,
\]
so that the Lie subalgebra $\Lambda^{-}\R^3$ is preserved under the flows $V_k$.
In particular, the recursion equation \eqref{eq:LaxY} can be rewritten as
\begin{equation}\label{eq:Lax0}
Y'+(\lambda Y)_{+} \times Y=0\,,
\end{equation}
which is the lowest flow $V_0$ of the hierarchy $V_k$ restricted to $\Lambda^{-}\R^3$. The Lie theoretic flows $V_k$ are related to the geometric flows $Y_k$ as follows:
\begin{The}\label{thm:Y2V}
If $\gamma\in\mM$ evolves by $Y_k$, that is, $\dot{\gamma}=Y_k$, then the generating loop $Y$ evolves by $V_k$, that is, $\dot{Y}=V_{k}$. 
\end{The}
\begin{proof}
We  derive an inhomogeneous ODE for $\dot{Y}$ which then can be solved by variation of the constants. The vector fields $V_0$ and $V_k$ commute, thus
\[
\dot{Y}'=(Y')\udot= \dot{Y}\times \lambda Y_0+Y\times\lambda \dot{Y_0}
\]
due to \eqref{eq:Lax0} and $Y_0= (\lambda Y)_{+}$. From the explicit form of $Y_k$ given in Theorem~\ref{thm:recursionsolution} it follows that the the reparametrization flow $Y_0$ commutes with all $Y_k$. Therefore,
\[
\dot{Y_0}=\dot{T}=(\gamma')\udot=\dot{\gamma}'=Y_k'=Y_{k+1}\times Y_0
\]
and we obtain the inhomogeneous ODE
\[
\dot{Y}'+\lambda Y_0\times \dot{Y}= (\lambda Y_0\times Y_{k+1})\times Y\,.
\]
Since $Y$ is a solution to the homogeneous equation, we make the variation of the constants Ansatz 
\[
\dot{Y}=[Y,Z]\,.
\]
Inserting this into the inhomogeneous ODE and using Jacobi identity yields
\[
(Z'+\lambda Y_0\times Z-\lambda Y_0\times Y_{k+1})\times Y=0\,,
\]
which is easily seen to be solved by $Z=(\lambda^{k+1}Y)_{+}$.  Thus,  the general solution of the inhomogeneous ODE is 
 $\dot{Y}=[Y,(\lambda^{k+1}Y)_{+}]+c Y$ for some $t$-independent $c$. But  $(Y,Y)=1$ by \eqref{eq:Yunit}, hence $c=0$ and $\dot{Y}= V_k$  as stated.
\end{proof}
We are now in a position to show the commutativity of the geometric flows $Y_k$.
\begin{The}
The flows $Y_k$ commute on $\mM$.
\end{The}
\begin{proof}
Considering the flows $\dot{\gamma}=Y_k$ and $\mathring{\gamma}=Y_l$, we need to verify that 
\[
\mathring{Y_k}=\dot{Y_l}\,.
\]
The previous theorem implies that the generating loop $Y=\sum_{k\geq 0}Y_k\lambda^{-k}$ satisfies 
\[
\dot{Y}=V_k\quad\text{and}\quad \mathring{Y}=V_l\,.
\]
Thus it suffices to show that the $\lambda^{-l}$ coefficient of $V_k=Y\times (\lambda^{k+1}Y)_{+}$ is equal to the $\lambda^{-k}$ coefficient of $V_l=Y\times (\lambda^{l+1}Y)_{+} $.  Unravelling this condition we need to check the equality
\[
\sum_{\substack{i+j=k+l+1\\ 0\leq j\leq k}} Y_i\times Y_j=\sum_{\substack{i+j=k+l+1\\ 0\leq j\leq l}} Y_i\times Y_j
\]
which indeed holds due to the skew symmetry of the vector cross product.
\end{proof}
 At this stage it is worth noting that we have provided a dictionary translating the geometrically defined flows $Y_k$ on $\mathcal{M}$ into the purely Lie theoretic defined flows $V_k$ on $\Lambda\R^3$. 
There is a rich solution theory of such equations and a particularly well studied class of solutions are the ``finite gap solutions'', which arise from finite dimensional algebro-geometric integrable systems.  

In much of what follows, we will need to complexify the loop Lie algebra
$\Lambda\R^3$ to $\Lambda\C^3$ allowing for complex values of the loop parameter $\lambda$. Depending on the context, the Lie algebra $\R^3={\bf su}_2$ will be identified with the Lie algebra of the special unitary group ${\bf SU}_2$. Likewise, $\C^3={\bf sl}_2(\C)$ will be identified with the Lie algebra of the special linear group ${\bf SL}_2(\C)$. The flows $V_k$ can be extended to $\Lambda\C^3$ and satisfy the reality condition $V_{k}(\bar{\xi})=\overline{V_{k}(\xi)}$. In particular, the vector fields $V_k$ preserve $\Lambda\R^3\subset \Lambda\C^3$. 
\begin{Lem}
The subspaces 
 \[
 \Lambda_d=\{\xi\in\Lambda^{-}\C^3\,;\, \xi=\sum_{k=0}^{d} \xi_k\lambda^{-k}\}\subset \Lambda\C^3
 \]
 of Laurent polynomials of degree $d\in\N$ are invariant under the commuting flows $V_k$ acting non-trivially only for $0\leq k\leq d-1$ on $\Lambda_d$.  In particular, for real initial data $\eta\in\Lambda_{d}^{\R}$ the flow
 \begin{equation}\label{eq:finitegapflow}
 d\xi=\sum_{k=0}^{d-1}V_k(\xi)dt_{k}
 \end{equation}
 has a real global solution $\xi\colon \R^{d}\to\Lambda_d^{\R}$ with $(\xi,\xi)=(\eta,\eta)$.
\end{Lem}
\begin{proof}
Since
\[
V_k=\xi\times (\lambda^{k+1}\xi)_{+}=-\xi\times (\lambda^{k+1}\xi)_{-}
\]
we see that $V_k$ are tangent to $\Lambda_d$. The invariant  inner product on $\Lambda\R^3$ restricts to a positive inner product on the finite dimensional space $\Lambda_d^{\R}$ which  $V_k$ preserves by Lemma~ \ref{lem:fundamental}. In particular, the flows evolve on a compact finite dimensional sphere and thus are complete.
\end{proof}
A fundamental invariant of the dynamics of $V_k$ is the hyperelliptic {\em spectral curve}
\begin{equation*}\label{eq:spectralcurve}
 \Sigma=\{(\lambda,\mu)\in\C^{\times}\times\C\,;\, \mu^2+(\xi,\xi)=0\}
 \end{equation*}
which only depends on the initial condition $\eta\in\Lambda_{d}^{\R}$ of a solution $\xi\colon\R^d\to \Lambda_d^{\R}$ of the flow \eqref{eq:finitegapflow}.  The spectral curve $\Sigma$ completes to a genus $d-1$ curve equipped with a real structure covering complex conjugation in $\lambda$.  Since 
$(\xi,\xi)=\Tr(\xi^2)$, the spectral curve encodes the constant eigenvalues of the ``isospectral'' solution 
$\xi$. Alternatively, one could consider for each $t\in\R^d$ the eigenline curve
\begin{equation}\label{eq:eigenlinecurve}
\Sigma_t=\{(\lambda,\C v)\,;\, \xi_{\lambda}(t) v=\mu_{\lambda} v\} \subset \C\times \CP^1\,,
\end{equation}
which can be shown to be biholomorphic to $\Sigma$. Pulling pack the canonical bundle over $\CP^1$ under $\Sigma_t\subset \C\times \CP^1$
gives a family of holomorphic line bundles $L_{t}\to\Sigma$.  The dynamics of the ODE \eqref{eq:finitegapflow} is encoded in the eigenline bundle flow $L\colon\R^d\to \Jac(\Sigma)$ of $\xi$ in the Jacobian $\Jac(\Sigma)$ of $\Sigma$.  The flow of line bundles $L_t$ is known to be linear and thus the vector fields  $V_k$ linearize on $\Jac(\Sigma)$.  This exhibits the finite gap theory as a classical integrable system with explicit formulas for solutions in terms of theta functions on the spectral curve $\Sigma$. 

Returning to our geometric picture, finite gap solutions have a purely variational description when viewed on the space of curves with fixed monodromy $\mM$.
Let $E_k\in\C^{\infty}(\mM,\R)$ be the putative Hamiltonians for the vector fields $Y_k$.
\begin{The}\label{thm:finitegap}
The following statements on a curve $\gamma\in\mM$  are equivalent:
\begin{enumerate}
\item
$\gamma$ is an  $E_{d+1}$ critical curve constrained by $E_1,\dots, E_d$.
\item
The lowest order flow
\[
\xi'=V_0(\xi)
\]
admits a solution $\xi=\sum_{k=0}^{d}\xi_k\lambda^{-k}$ on $\Lambda_d^{\R}$ with $(\xi,\xi)=1$. 
\end{enumerate}
The relation between the curve $\gamma$ and the solution $\xi$ is given by $\gamma'=\xi_0$. 
\end{The}
\begin{proof}
From \eqref{eq:criticalY}  we know that $\gamma$ is $E_{d+1}$ constrained critical if and only if 
\[
Y_{d+1}=\sum_{k=0}^{d}c_kY_k
\]
for constants $c_i\in\R$.  Setting 
\begin{equation}\label{eq:xi2Y}
\xi_k=Y_k-c_d Y_{k-1}-\cdots - c_{d-k+1}Y_0
\end{equation}
for $k=0,\dots,d$ provides a Laurent polynomial solution to $\xi'=V_0(\xi)$ due to the recursion relation
$Y_k'+Y_0\times Y_{k+1}=0$ for $Y_k$ in \eqref{eq:prerecursion}. From \eqref{eq:Yunit} we know that $(Y,Y)=1$ and we also have  $\xi_0=Y_0=T=\gamma'$.  To show the converse, we start from a Laurent  polynomial solution $\xi=\sum_{k=0}^{d}\xi_k\lambda^{-k}$ and verify the relations \eqref{eq:xi2Y} inductively starting from $Y_0=\xi_0$. Since $\xi_{d+1}=0$, this gives the desired constrained criticality relation on $Y_{d+1}=\sum_{k=0}^{d}c_kY_k$. 
\end{proof}
Theorem~\ref{thm:finitegap} provides a fairly explicit recipe 
\cite{calini2000,calini2004}, \cite{matsutani2016},
for the construction of all $E_{d+1}$ critical  constrained curves $\gamma\in\mM$  with a given monodromy, together with their isospectral deformations by the higher order flows $Y_k$ for $k=1,\dots, d-1$. Since the  Lax flows  $d\xi=\sum_{k=0}^{d-1}V_k(\xi)dt_{k}$ are linear flows on the Jacobian of the spectral curve $\Sigma$, the solution $\xi$ can in principle be expressed in terms of theta functions on $\Sigma$. The curve $\gamma$ is then obtained from $\xi_0=\gamma'$ with $t_0=x$ and the ``higher times'' $t_k$ accounting for the isospectral deformations. Periodicity conditions, for instance, the required monodromy of the resulting curve $\gamma$, can be discussed via the Abel--Jacobi map. A classical example is given by Euler elastica, that is, curves $\gamma\in\mM$ with a given monodromy critical for the bending energy $E_3$ under length and torsion constraints: from our discussion we know that they can  be computed explicitly from elliptic spectral curves. 
 
\section{Associated family of curves}
The remaining question to address is why the commuting hierarchy $Y_k$ of vector fields on $\mM$ are symplectic for Hamiltonians $E_k\colon \mM\to\R$.  This requires a closer look at the geometry behind the generating loop $Y=\sum_{k\geq 0}Y_k\lambda^{-k}$ where  $Y_0=T=\gamma'$.  Lemma~\ref{lem:fundamental} implies that the generating loop $Y$, which satisfies the  Lax equation
\[
Y'+\lambda T \times Y=0\,,
\]
is given by 
\begin{equation*}\label{eq:adjointY}
Y=F^{-1}Y(x_0)F
\end{equation*}
where $F\colon M\times\C\to{\bf SL}_2(\C)$ solves the linear equation
\begin{equation}\label{eq:Fode}
 F'=F \lambda T\,,\,\,\, F(x_0)=\one\,.
 \end{equation}
 Here $x_0\in M$ denotes an arbitrarily chosen base point. Even though $Y$ is only a formal loop, $F$ is smooth on $M$, holomorphic in $\lambda\in\C$, and takes values in ${\bf SU}_2$ for real values $\lambda\in\R$.  In other words, we can think of  $F\colon M\to \Lambda^{+}{\bf SL}_2(\C)$ as a map into the loop group of holomorphic loops into ${\bf SL}_2(\C)$ equipped with the reality condition
 $
 F_{\bar{\lambda}}=({F_{\lambda}}^{*}) ^{-1}
 $.
 Note that \eqref{eq:Fode} implies that $F_{0}=\one$ identically on $M$, which is to say that $F$ maps into the based (at $\lambda=0$) loop group.  
 \begin{Def}\label{def:associatedfamily}
 Let $\gamma\in\mM$ be a curve with monodromy. The {\em associated family} of curves $\gamma_{\lambda}\colon M\to \R^3$ is given by 
 \[
 \gamma_{\lambda}'=F_{\lambda}TF_{\lambda}^{-1} \,,\,\,\, \gamma_{\lambda}(x_0)=\gamma(x_0)
 \]
 for real $\lambda\in\R$. 
 \end{Def}

\begin{figure}[ht]
  \centering
  \includegraphics[width=.6\columnwidth]{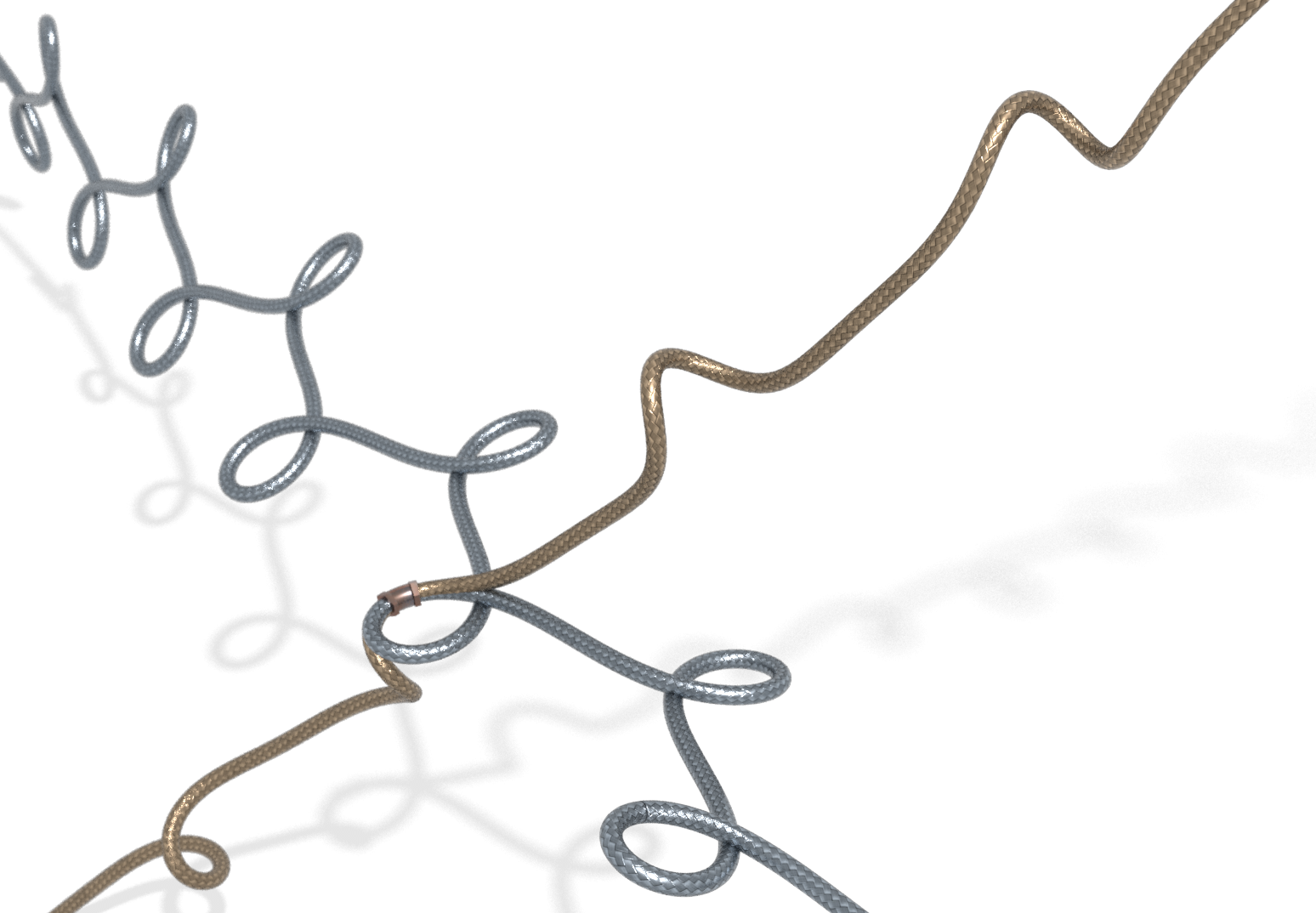}
  \caption{A curve $\gamma$ (blue) together with a member $\gamma_{\lambda}$ (gold) of its associated family for real $\lambda$. The two curves osculate to second order at $\gamma(x_0)$. 
  \label{fig:AssociateFamily}}
\end{figure}

Since $F_{\lambda=0}=\one$, the associated family $\gamma_{\lambda}$  is a deformation
of the initial curve $\gamma=\gamma_{0}$ osculating to first order at $\gamma(x_0)$. In fact, $\gamma_{\lambda}$ osculates to second order as can be seen by calculating the complex curvatures $\psi_{\lambda}\colon M\to\C$ of the curves $\gamma_{\lambda}$, which agree at $x_0\in M$ as shown in Figure~\ref{fig:AssociateFamily}. The complex curvature $\psi\colon M\to\C$ of a curve $\gamma\in\mM$ is given by $\gamma''=\psi\xi$ for a unit length parallel section $\xi\in\Gamma(\perp_{\gamma}\!\!M)$ of the normal bundle, a unitary complex line bundle. 
It has been shown \cite{hasimoto1972} that the complex curvature evolves under the non-linear Schr\"odinger equation, if the corresponding curve evolves under the vortex filament flow.

In general, the curves $\gamma_{\lambda}$ have changing monodromies and thus are not contained in the original $\mM$. 
As we shall see, the dependency of these monodromies on the spectral parameter $\lambda$ gives a geometric realization of the spectral curve and, at the same time, provides the Hamiltonians $E_k$ for the flows $Y_k$.

\begin{Lem}\label{lem:associatedfamily}
Let $\gamma\in\mM$ be a curve with monodromy $\tau^*\gamma= A \gamma A^{-1}+a$ where $A\in{\bf SU}_2$ and $a\in\R^3$.
Then we have:
\begin{enumerate}
\item
The associated family $\gamma_{\lambda}$ can be expressed via the Sym formula
\[
\gamma_{\lambda}-\gamma(x_0)=\frac{dF_{\lambda}}{d\lambda}F_{\lambda}^{-1}\,.
\]
\item
The rotation monodromy $\tilde{A}_{\lambda}$  of $\gamma_{\lambda}$ is given by
\[
\tilde{A}=(\tau^*F)(x_0)A\in\Lambda^{+}{\bf SU}_2
\]
and due to $F_{0}=\one$, we have $\tilde{A}_0=A$. The translation monodromy of $\gamma_{\lambda}$ becomes
\[
a_{\lambda}=\frac{d\tilde{A}_{\lambda}}{d\lambda}\tilde{A}_{\lambda}^{-1} - \tilde{A}_{\lambda} \,\gamma(x_0)\tilde{A}_{\lambda}^{-1}\,.
\]
\item
The total torsion $E_2$ of $\gamma_{\lambda}$ is given by 
\[
E_2(\gamma_{\lambda})=E_2(\gamma)+\lambda E_1(\gamma)
\]
with $E_1(\gamma)$ the length of $\gamma$ over a fundamental domain $I\subset M$. 
\end{enumerate}
\end{Lem}
\begin{proof}
To check the Sym formula, we calculate
\[
\gamma_{\lambda}'=(\frac{dF_{\lambda}}{d\lambda}F_{\lambda}^{-1})'=\frac{dF_{\lambda}'}{d\lambda}F_{\lambda}^{-1}=F_{\lambda}TF_{\lambda}^{-1}\,.
\]
Since $F(x_0)=1$ we also have $\frac{dF(x_0)}{d\lambda}F^{-1}(x_0)=0$, which verifies $\gamma_{\lambda}(x_0)=\gamma(x_0)$.
From the Sym formula we read off the monodromy of the curve $\gamma_{\lambda}$ as stated. To compute the formula for $\tilde{A}$, we observe that the ODE \eqref{eq:Fode} has $\tau^*FA$ as a solution since $\gamma$ has rotational monodromy $\tau^*T=ATA^{-1}$. Therefore, $\tau^* FA= \tilde{A}F$ for some $\tilde{A}\in\Lambda^{+} {\bf SU}_2$ and evaluation at $x_0\in M$ gives $\tilde{A}= (\tau^*F)(x_0)A$.  For the total torsion of $\gamma_{\lambda}$ we use its characterization in Lemma~\ref{lem:torsion}:
if $\nu \in T_{\gamma}\mM_A$ is a normal vector field along $\gamma$, then $\nu_{\lambda}=F_{\lambda}\nu F_{\lambda}^{-1}\in T_{\gamma_{\lambda}}\mM_{\tilde{A}_{\lambda}}$ is a normal field along $\gamma_{\lambda}$. Therefore,  
\begin{align*}
E_2(\gamma_{\lambda})&=\int_I(\nabla \nu_{\lambda}, T_{\lambda}\times \nu_{\lambda})=\int_{I} (\nabla T +\lambda T\times \nu dx, T\times \nu )\\
&=E_2(\gamma)+\lambda E_1(\gamma)
\end{align*}
where we used $\nabla T_{\lambda}=F_{\lambda}(\nabla T+\lambda T\times \nu dx)F_{\lambda}^{-1}$ and the invariance of the inner product.
\end{proof}

\begin{figure}[ht]
\centering
\includegraphics[width=.7\columnwidth]{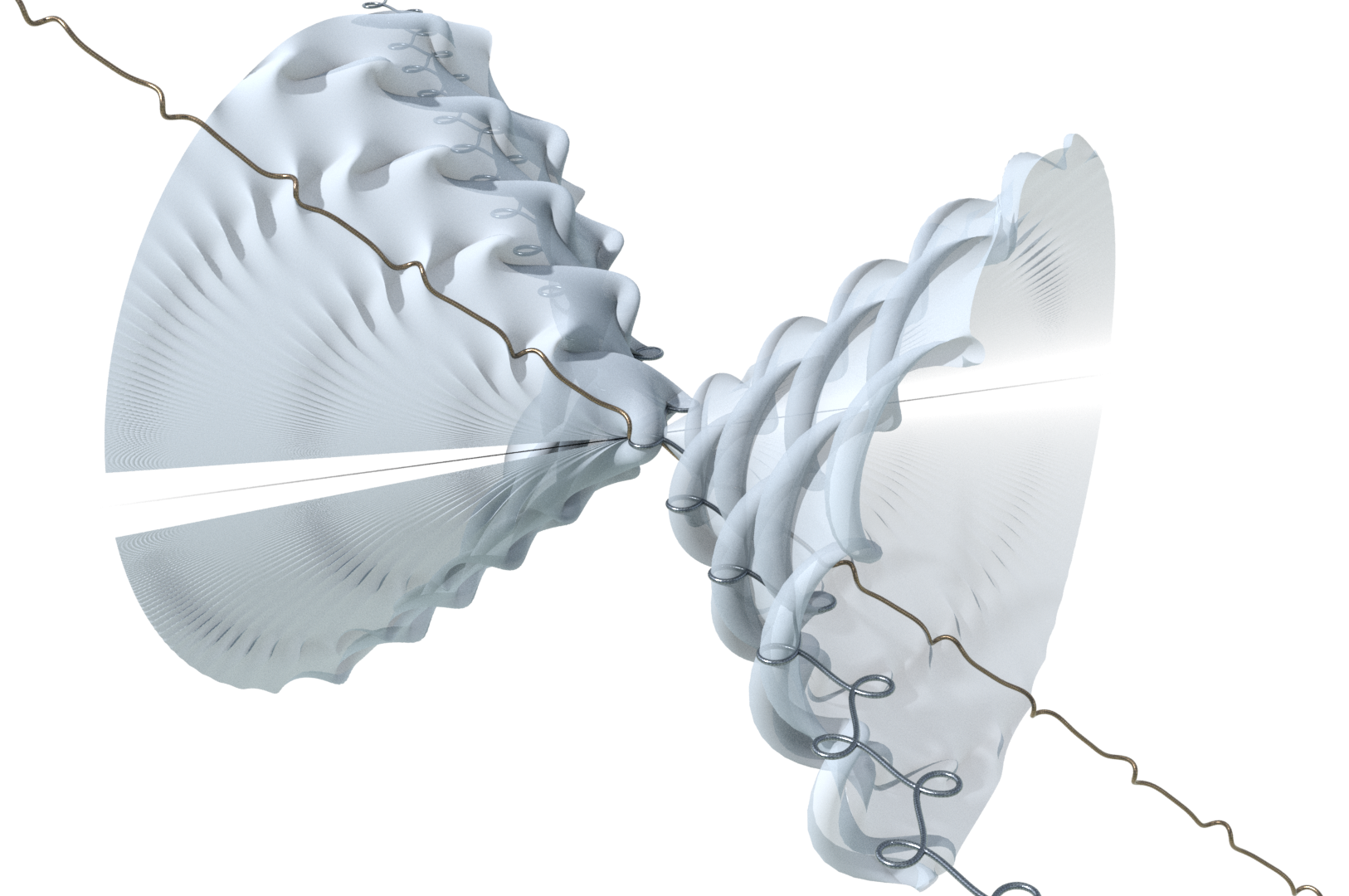}
\caption{The curves $\gamma_{\lambda}$ in the associated family of an elastic curve sweep out a surface that is close to the quadratic cone swept out by their monodromy axes. The intersection of this cone with a sphere centered at $\gamma(x_0)$ is a quartic curve in bijective algebraic correspondence with the real part of the spectral curve $\Sigma$, which is elliptic  in this case. As the spectral parameter $\lambda\to\infty$, the curves $\gamma_{\lambda}$ straighten, as can be seen when approaching the white sector. 
\label{fig:SpectralCurveCone}}
\end{figure}

Since the flows $Y_k$ are genuine vector fields on $\mM$, they satisfy $\tau^* Y_k=AY_kA^{-1}$ along $\gamma\in\mM$. Thus, the generating loop $Y=\sum_{k\geq 0}Y_k\lambda^{-k}$ along a curve $\gamma\in\mM$ is also invariant under the monodromy  $\tau^*Y=AYA^{-1}$. Therefore, 
\begin{equation}\label{eq:Ycommute}
Y(x_0)=\tau^*(Y(x_0))=\tau^*(FYF^{-1})=\tilde{A}Y(x_0)\tilde{A}^{-1}
\end{equation}
which implies
that the axis of the rotation monodromy $\tilde{A}$ is in the direction of $Y(x_0)$.  Ignoring for the moment that
$Y(x_0)\in\Lambda^{-}\R^3$ is only a formal Laurent series, we can express the  monodromy
\[
\tilde{A} =\exp(-\frac{1}{2}\theta Y(x_0))
\]
by its unit length axis and rotation angle $\theta$.  Note that $\tilde{A}=\tilde{A}(x_0)$ depends on the base point $x_0\in M$ 
by Lemma~\ref{lem:associatedfamily}. In particular, as one moves the base point
\begin{equation}\label{eq:basechange}
\tilde{A}(x)=F(x)^{-1}\tilde{A}F(x)\,,\,\,\, x\in M
\end{equation}
changes by conjugation. This has the important implication that the monodromy angle $\theta$ is independent of $x_0\in M$. Therefore, $\theta$ depends only on the curve $\gamma\in\mM$ and we obtain for each real $\lambda\in\R$  a (modulo $2\pi$) well defined function 
\[
\theta_{\lambda}\colon \mM\to \R\,.
\]
The monodromy angle function $\theta$ of the associated family  turns out to be the Hamiltonian for the generating loop $Y=\sum_{k\geq 0}Y_k\lambda^{-k}$:
\begin{The}\label{thm:sY}
Let $\gamma_t\in \mM$ be a variation of $\gamma=\gamma_0$ with rotational monodromy $A\in{\bf SU}_2$, then
\[
\tfrac{1}{\lambda^{2}}\dot{\theta}=\tfrac{1}{\lambda^{2}} d \theta_{\gamma}(\dot{\gamma})=\sigma(Y,\dot{\gamma})\,.
\]
In other words, $Y$ is the symplectic vector field for the Hamiltonian $\lambda^{-2}\theta$.
\end{The}
Before continuing, we need to remedy the problem that the generating loop $Y$ of the genuine vector fields $Y_k$ is only a formal series.
From  \eqref{eq:Ycommute} we know that $Y(x_0)$ is in the direction of the axis of the rotation $\tilde{A}$, which is given by the trace-free part
of $\tilde{A}$. Since $\tilde{A}\in\Lambda^{+}{\bf SU}_2$ is a real analytic loop, we may choose a unit length real analytic map $\tilde{Y}(x_0)\colon \R\to S^2\subset \R^3$ spanning the axis of $\tilde{A}$, that is
\[
\R \tilde{Y}(x_0)_{\lambda}=\R (\tilde{A}_{\lambda}-\tfrac{1}{2} \Tr \tilde{A}_{\lambda}\one)\,.
\]
 From \eqref{eq:basechange} we deduce $\tilde{Y}=F^{-1}\tilde{Y}(x_0)F$ and therefore $\tilde{Y}$ satisfies the same Lax equation \eqref{eq:LaxY} as the formal loop $Y$. It can be shown by asymptotic analysis \cite[Chapter~I.3]{faddeev1987} that $\tilde{Y}_{\lambda}\to T$ as $\lambda\to \infty$  in the smooth topology over a fundamental domain $I\subset M$. Since $T_{\lambda}=F_{\lambda}TF_{\lambda}^{-1}$, we also obtain that $T_{\lambda}\to T(x_0)$ as $\lambda\to\infty$. In particular, the curves $\gamma_{\lambda}$ in the associated family straighten out and tend to the line with tangent $T(x_0)$, as can be seen in Figure~\ref{fig:SpectralCurveCone}. Knowing that $\tilde{Y}$ is smooth at $\lambda=\infty$, it has the Taylor series expansion $\tilde{Y}=\sum_{k\geq 0}\tilde{Y}_k\lambda^{-k}$ and satisfies the recursion equation \eqref{eq:LaxY} with the same initialization $\tilde{Y}_0=Y_0=T$. But then the explicit construction of the solutions to this recursion in Theorem~\ref{thm:recursionsolution} shows that $Y_k=\tilde{Y}_k$ for $k\geq 0$. 
In other words, instead of working with $Y$ in the construction of the angle function $\theta$, we need to work with $\tilde{Y}$, and likewise in any of the arguments to come. We will still keep writing $Y$, but think $\tilde{Y}$, as not to stray too far from our geometric intuition. 

With this being said, 
we come to the proof of Lemma~\ref{thm:sY}.
\begin{proof}
We first derive an expression for the variation $\dot{\tilde{A}}$ of the monodromy. Writing  $\dot{F}=HF$, then $F'=F(\lambda T)$, $F(x_0)=\one$ implies that 
\[
H'=F\lambda \dot{T}F^{-1}\,,\,\,\, H(x_0)=0\,.
\]
From $\tilde{A}=(\tau^*F)(x_0)A$ we obtain
\begin{align*}
\dot{\tilde{A}}&=(\tau^*\dot{F})(x_0)A=(\tau^*H)(x_0)(\tau^*F)(x_0)A=(\tau^*H)(x_0)\tilde{A}\\
&=\int_I F\lambda\dot{T}F^{-1}\tilde{A} dx=\int_I F\lambda\dot{T}\tilde{A}(x)F^{-1}dx
\end{align*}
where we used \eqref{eq:basechange} and the fundamental domain $I\subset M$  has oriented boundary $\del I=\{x_0\}\cup \{\tau(x_0)\}$. Since $\tilde{A}(x)=\exp(-\tfrac{1}{2}\theta Y(x))=\cos(\tfrac{1}{2}\theta)\one-\sin(\tfrac{1}{2}\theta)\,Y(x)$ with $Y(x)$  trace free, we arrive at 
\begin{align*}
-\dot{\theta}\sin(\tfrac{1}{2}\theta) &=\Tr \dot{\tilde{A}}=\int_I \Tr(\lambda \dot{T} \tilde{A}(x))dx=\int_I \Tr(\lambda \dot{T} (\cos(\tfrac{1}{2}\theta)\one -
\sin(\tfrac{1}{2}\theta)\, Y))\,dx\\
&=-\sin(\tfrac{1}{2}\theta) \int_I(\lambda \dot{T},Y)dx\,.
\end{align*}
Finally, we cancel $\sin(\tfrac{1}{2}\theta)$ on both sides, replace $\dot{T}$ with $\dot{\gamma}$ using \eqref{eq:dotgamma}, and integrate by parts to get
\begin{align*}
\dot{\theta}&=\int_I (\lambda \dot{T} , Y)dx=-\int_I(\lambda \dot{\gamma}, Y')dx=\int_I(\lambda \dot{\gamma}, \lambda\,T\times Y)dx\\
&=\sigma(\lambda^2 Y, \dot{\gamma})
\end{align*}
where we also used \eqref{eq:LaxY}.
\end{proof}
It remains to calculate an explicit expression for the Taylor expansion 
\begin{equation}
\tfrac{1}{\lambda^2}\theta=\sum_{k\geq 0}E_k\lambda^{-k}
\label{thetaexpansion}
\end{equation}
of the monodromy angle $\theta$ at $\lambda=\infty$ to obtain concrete formulas for the hierarchy of Hamiltonians $E_k\colon\mM\to\R$ from Theorem~\ref{thm:sY}. Starting with a curve $\gamma\in\mM$ with rotation monodromy $A$, we consider its associated family of curves $\gamma_{\lambda}$ with monodromy $\tilde{A}_{\lambda}$ at the base point $x_0\in M$.  The tangent image $T_{\lambda}\colon M\to S^2\subset \R^3$ defines the following sector on $S^2$, see Figure~\ref{fig:GaussBonnet}:
\begin{enumerate}
\item
The monodromy axis point $Y_{\lambda}(x_0)$ connected to $T_{\lambda}(x_0)$ by a geodesic arc, then the curve $T_{\lambda}$ traversed from $x_0$ along a fundamental domain $I\subset M$ to  $\tau(x_0)$, and the geodesic arc connecting $T_{\lambda}(\tau(x_0))$ back to the  axis point $Y_{\lambda}(x_0)$. Since $Y_{\lambda}\to T$ tends to  $T$ and $T_\lambda\to T(x_0)$  tends to  $T(x_0)$  for large $\lambda\in\R$, there is no ambiguity in this prescription.  
\item
The exterior angles of this sector are $\pi-\theta_{\lambda}$ at $Y_{\lambda}(x_0)$,  some angle $\beta$ at $T_{\lambda}(x_0)$, and the angle $\pi-\beta$ at $T_{\lambda}(\tau(x_0))$, since $T_{\lambda}$ has rotation monodromy around the axis $Y_{\lambda}(x_0)$ with angle $\theta_{\lambda}$. 
\end{enumerate}
\begin{figure}[ht]
  \centering
  \includegraphics[width=.4\columnwidth]{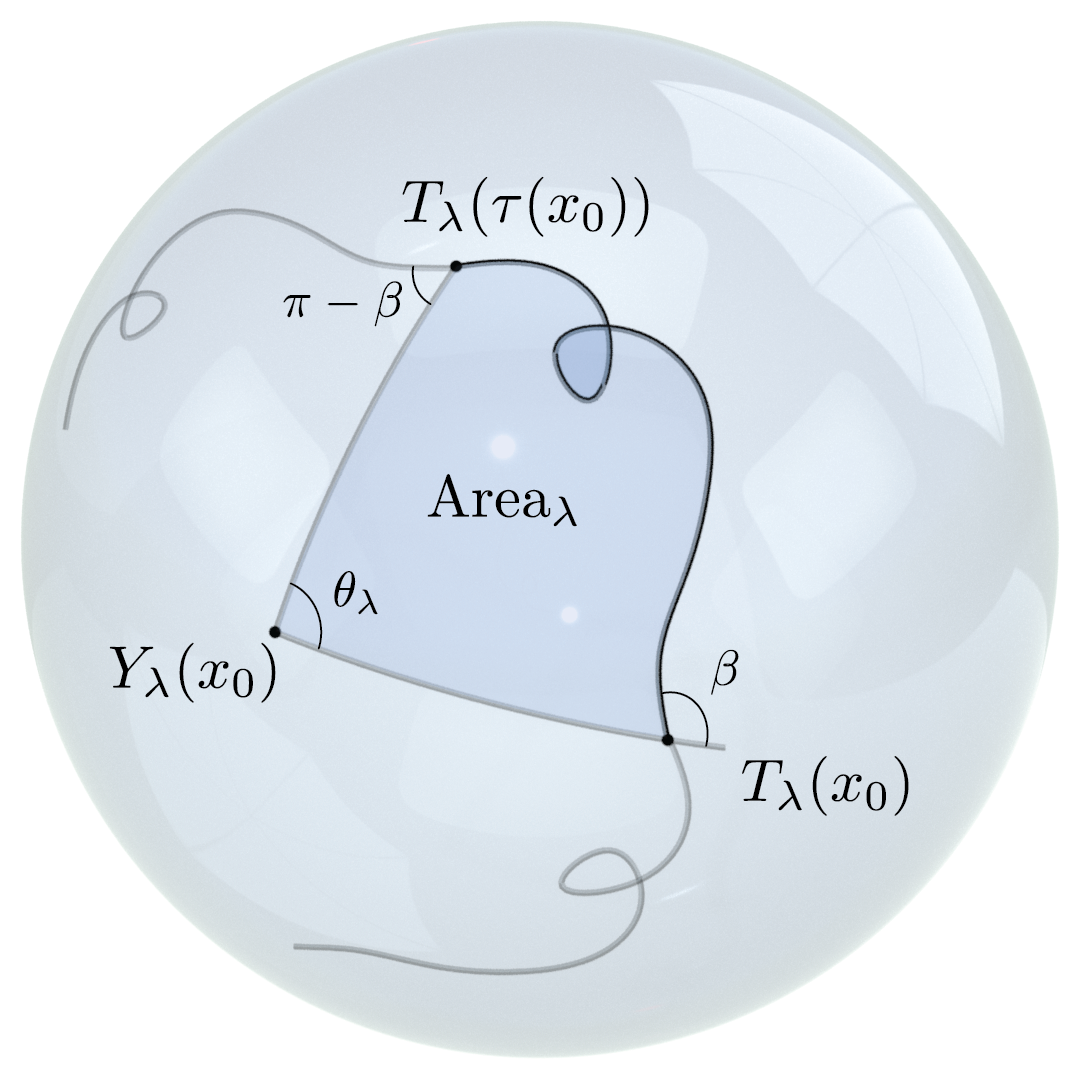}
  \caption{The monodromy angle $\theta_\lambda$ can be computed in terms of the area of the sector enclosed by the tangent image $T_\lambda$.
  \label{fig:GaussBonnet}}
\end{figure}

Parallel transport in the normal bundle $\perp_{\gamma_{\lambda}}\!\!M$ along  $\gamma_{\lambda}$ is the same as parallel transport with respect to the Levi-Civita connection along the tangent image $T_{\lambda}$. Applying the Gau{\ss}--Bonnet Theorem to the parallel transport around our sector of area $ \text{Area}_{\lambda}$ gives the relation
\[
E_2(\gamma_{\lambda})+(\pi-\theta_{\lambda})+\beta + (\pi-\beta)+ \text{Area}_{\lambda}=2\pi\,.
\]
where we used Lemma~\ref{lem:torsion}. Applying Lemma~\ref{lem:associatedfamily} (iii), this unravels to 
\begin{equation*}\label{eq:GaussBonnet}
\theta_{\lambda}=\lambda E_1(\gamma) + E_2(\gamma)+ \text{Area}_{\lambda}\,.
\end{equation*}
The area of a spherical sector of the type above is given by 
\[
\text{Area}_{\lambda}= \int_I\frac{\det(Y_\lambda(x_0), T_\lambda, T_{\lambda}')}{1+(Y_\lambda(x_0),T_{\lambda})}\, dx
\]
which we can further simplify: from Definition~\ref{def:associatedfamily} of the associated family, we have $T_{\lambda}=F_{\lambda} T F_{\lambda}^{-1}$. Furthermore, $Y$ as a solution to the Lax equation \eqref{eq:LaxY} satisfies $Y=F^{-1}Y(x_0)F$,  which leads to 
\[
\text{Area}_{\lambda}= \int_I\frac{\det(Y, T, T')}{1+(Y,T)}\, dx\,.
\]
To calculate $(Y,T)=\sum_{k\geq 0}(Y_k,T)\lambda^{-k}$, we need the tangential components of $Y_k$, which can be read off from the proof of Theorem~\ref{thm:recursionsolution}:
\[
(T,Y_0)=1, \quad (T,Y_1)=0 ,\quad (T,Y_k) =-\tfrac{1}{2}\sum_{i=1}^{k-1}(Y_{k-i},Y_{i})\,.
\]
Expressing $\tfrac{1}{1+(T,Y)}$ via the geometric series, Theorem~\ref{thm:sY} gives rise to explicit formulas for the Hamiltonians $E_k$.
\begin{The}\label{Hamiltonians}
The Hamiltonians $E_k$ for the commuting flows $Y_k$ are given by the generating function
\[
\sum_{k\geq 0}E_k\lambda^{-k}=0+E_1\lambda^{-1}+E_2\lambda^{-2} + \tfrac{1}{2} \int_I \sum_{i\geq 1, j\geq 0} (Y_i, Y_1)\lambda^{-i}(-1)^j(\sum_{l\geq 2} \tfrac{1}{2}(T,Y_l)\lambda^{-l})^j\, dx\,.
\]
For instance, 
\begin{itemize}
\item
$E_4=- \tfrac{1}{2}\int_I\det(\gamma',\gamma'',\gamma''')dx$,
\item
 $E_5=\int_I (\tfrac{1}{2}|\gamma'''|^2-\tfrac{5}{8}|\gamma''|^4)\,dx$,
\item
$E_6=\int_I (-\tfrac{1}{2}\det(\gamma',\gamma''',\gamma'''')+\tfrac{7}{8}|\gamma''|^2\det(\gamma',\gamma'',\gamma'''))\,dx$.
\end{itemize}
\end{The}

\section{Darboux transforms and the spectral curve revisited}
The geometry of the associated family $\gamma_{\lambda}$ of a curve $\gamma\in\mM$ played a pivotal role in our discussions so far. For instance, from \eqref{eq:eigenlinecurve} the real 
part of the spectral curve $\Sigma$ is given by the eigenlines of the monodromy $A_{\lambda}$ for real spectral parameter $\lambda\in\R$.  A natural question to ask is whether there are associated curves $\gamma_{\lambda}$ for complex spectral parameters $\lambda\in\C$.  This is indeed the case, provided that we allow the curves $\gamma_{\lambda}$ to live as curves with monodromy in hyperbolic 3-space $H^3$.  The two
 fixed points of their monodromies on the sphere at infinity of $H^3$  exhibit the spectral curve $\Sigma$ as a hyperelliptic branched cover over the complex $\lambda$-plane. Moreover, these fixed points give rise to periodic Darboux transforms $\eta\in\mM$  of the original curve $\gamma$ as curves in $\mathbb{R}^3$. Thus, we have associated to a curve $\gamma\in\mM$ of a given monodromy $h$ a Riemann surface $\Sigma$ worth of curves $\eta\in\mM$, the Darboux transforms of $\gamma$ of the same monodromy $h$. 
 \begin{figure}[ht]
  \centering
  \includegraphics[width=.8\columnwidth]{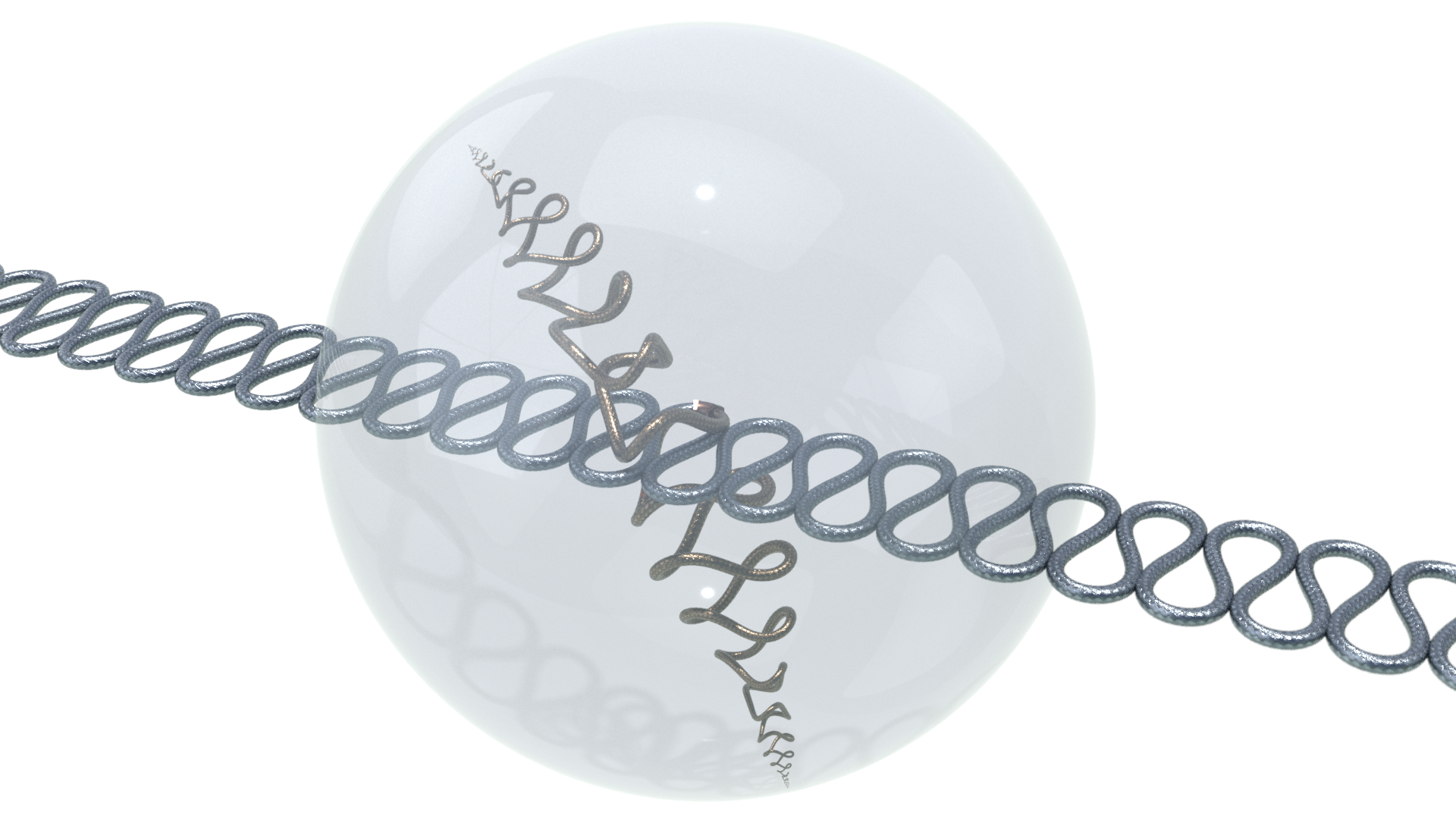}
  \caption{A curve $\gamma$ (blue) together with a member $\gamma_{\lambda}$ (gold) of its associated family for complex $\lambda$ in hyperbolic $3$-space.
  \label{fig:HyperbolicAssociateFamily}}
\end{figure}
 Given $\gamma\in\mM$ and using the notation from the previous section, we have a solution $F_{\lambda}\colon M\to {\bf SL}_2(\C)$ of 
 $F'=F(\lambda T)$, $F(x_0)=\one$, for all $\lambda\in\C$.  Using the description of hyperbolic space $H^3$ as the set of all hermitian $2\times 2$-matrices with determinant one and positive trace,  we define the associated family 
 \[
 \gamma_\lambda\colon M \to H^3\,,\quad \gamma_\lambda =F_{\lambda}F_{\lambda}^*
 \]
 for non-real values $\lambda\in\C\setminus \R$.  Since $\tau^*F A=\tilde{A} F$ with $\tilde{A}\in\Lambda^{+}{\bf SL}_2(\C)$ for non-real $\lambda$, 
 the curves $\gamma_{\lambda}$ in hyperbolic space $H^3$ have monodromy
 \[
 \tau^*\gamma_{\lambda} =\tilde{A}_{\lambda}\gamma_{\lambda} \tilde{A}_{\lambda}^*\,.
 \]
 Furthermore, the curves $\gamma_{\lambda}$ have constant speed  $2\Im(\lambda)$, as can be seen from
 \[
 \gamma_\lambda'=F(\lambda T +\bar{\lambda}T^*)F^*=2\Im(\lambda) iFTF^*\,.
 \]
 Stereographic projection 
 \[
 \pi\colon H^3 \to {\bf su}_2=\mathbb{R}^3\,,\quad \pi(p)= i\,\frac{\Tr p\,\one-2 p}{2+\Tr p}
 \]
 realizes  $H^3$ as the Poincare model $B^3$, the unit ball centered at the origin. Since $d_{\one}\pi(T)=\frac{T}{2}$, we rescale $B^3$ by $\frac{1}{\Im \lambda}$ and center it at $\gamma(x_0)$. Then the curve
\[
\hat{\gamma}_\lambda=\gamma(x_0)+\frac{1}{\Im \lambda} \pi \circ \gamma_\lambda
\]
in $\R^3$ touches $\gamma$ to first order at $\gamma(x_0)$, see Figure~\ref{fig:HyperbolicAssociateFamily}.

Viewed as hyperbolic motions, the monodromy matrices $\tilde{A}_{\lambda}$  have two fixed points on the sphere at infinity. Stereographic projection $\pi$ realizes these fixed points as two unit vectors $S_\pm\in S^2$. The set of all pairs $(\lambda,S_{\pm})\in \mathbb{C}\times S^2=\mathbb{C}\times \CP^1$ is biholomorphic to the eigenline spectral curve \eqref{eq:eigenlinecurve}. This means that as $\lambda$ varies, the pairs $(S_+,S_-)$ of points on $S^2$ trace out a Euclidean image of the spectral curve, see Figure~\ref{fig:ComplexSpectralCurve}.
\begin{figure}[ht]
  \centering
  \includegraphics[width=.5\columnwidth]{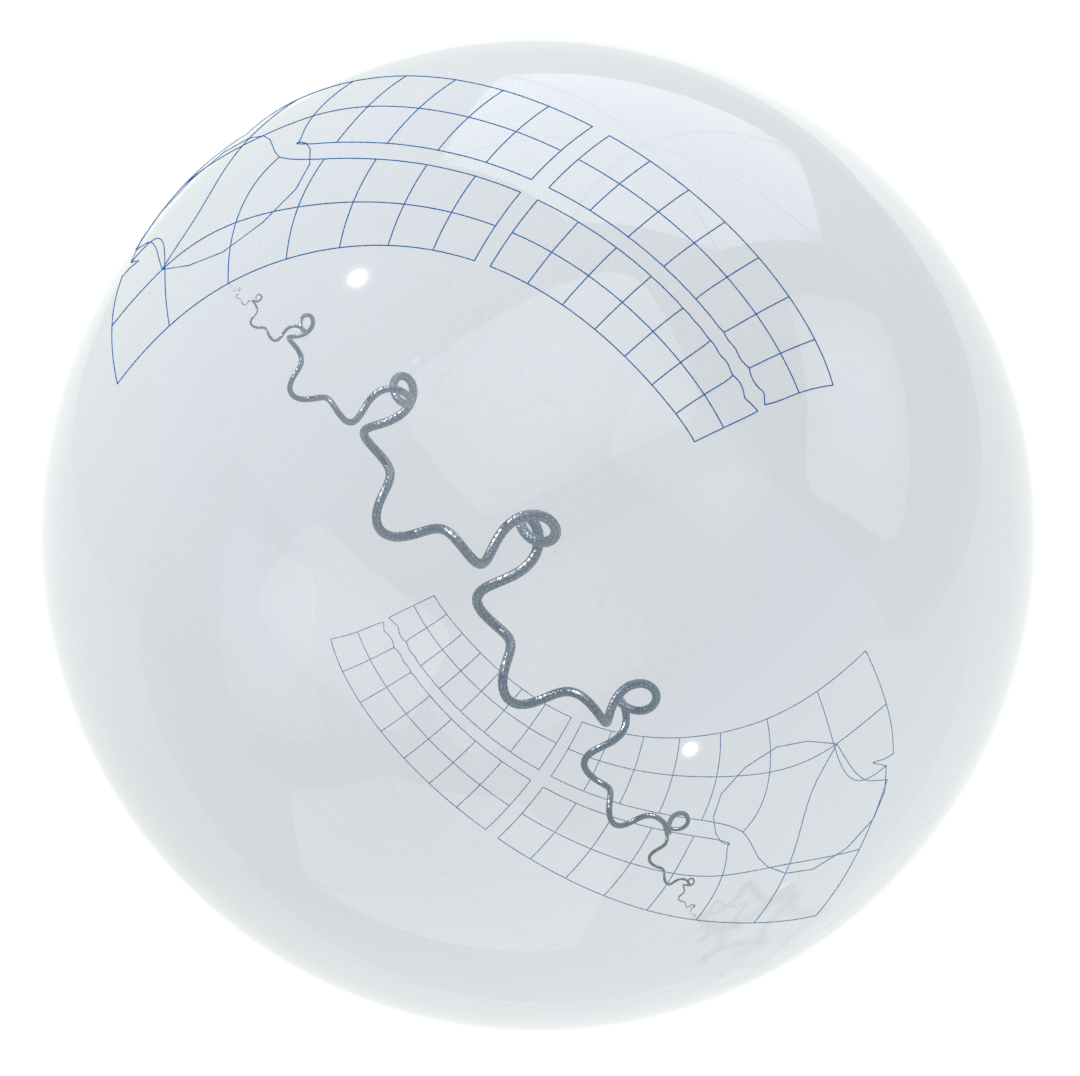}
  \caption{The part of the spectral curve that lies over a rectangular domain in the $\lambda$-plane.
  \label{fig:ComplexSpectralCurve}}
\end{figure}

It turns out that a slightly different scaling of the hyperbolic bubbles around the points $\gamma(x)$, which agrees with the above in case $\lambda\in i\R$ is purely imaginary, is closely related to {\em Darboux transforms} in the sense of \cite{hoffmann2008,pinkall2007} of $\gamma\in\mM$, see Figure~\ref{fig:Darboux}:

\begin{The}
Let $\gamma\in\mM$ then the two curves $\eta_{\pm}\colon M\to \R^3$ given by
\[
\eta_{\pm}= \gamma+\frac{\Im(\lambda)}{|\lambda|^2}S_{\pm}
\]
are Darboux transforms of $\gamma$ having the same monodromy as $\gamma$, that is $\eta_{\pm}\in\mM$. In particular, $\eta_{\pm}$ have constant distance to $\gamma$ and induce the same arclength on $M$ as $\gamma$. All Hamiltonians  $E_k$,  $k\geq -2$, satisfy 
\[
E_k(\eta_{\pm})=E_k(\gamma)\,.
\]
\end{The}
\begin{proof}
 For real $\lambda$, as we move the base point along $M$, equation \eqref{eq:basechange} implies that $S\pm$ satisfy the differential equation 
\[
S_{\pm}'=-\lambda T \times S_{\pm}\,.
\]
Viewing $S\mapsto T\times S$ as a vector field on $S^2$, we see that an imaginary part of $\lambda$ adds a multiple of the same vector field rotated by $\pi/2$:
\[
S_{\pm}'=-\Re(\lambda) T\times S_{\pm} - \Im(\lambda) T\times (T\times S_{\pm})\,.
\]
According to equation (25) of \cite{pinkall2007} (note that the letters $S$ and $T$ are interchanged there) this is precisely the equation needed for
\[
\eta_{\pm}:= \gamma(x_0)+\frac{\Im(\lambda)}{|\lambda|^2}S_{\pm}
\]
to be Darboux transforms of $\gamma$. Darboux transforms exhibit the so-called Bianchi permutability: Darboux transforming $\gamma$ with parameter $\lambda$ followed by a Darboux transform with parameter $\mu$ has the same result as the same procedure with the roles of $\lambda$ and $\mu$ interchanged. This was proved in a discrete setting in \cite{pinkall2007} and follows by continuum limit for the smooth case. The differential equation that determines Darboux transforms is the same as the one that determines the monodromies of the curves in the associated family. Therefore, as a consequence of permutability, $\eta_{\pm}$ has the same monodromy angle function $\theta$ as $\gamma$. Thus, by \eqref{thetaexpansion} we have $E_k(\eta_{\pm})=E_k(\gamma)$.
\end{proof}
\begin{figure}[ht]
  \centering
  \includegraphics[width=.8\columnwidth]{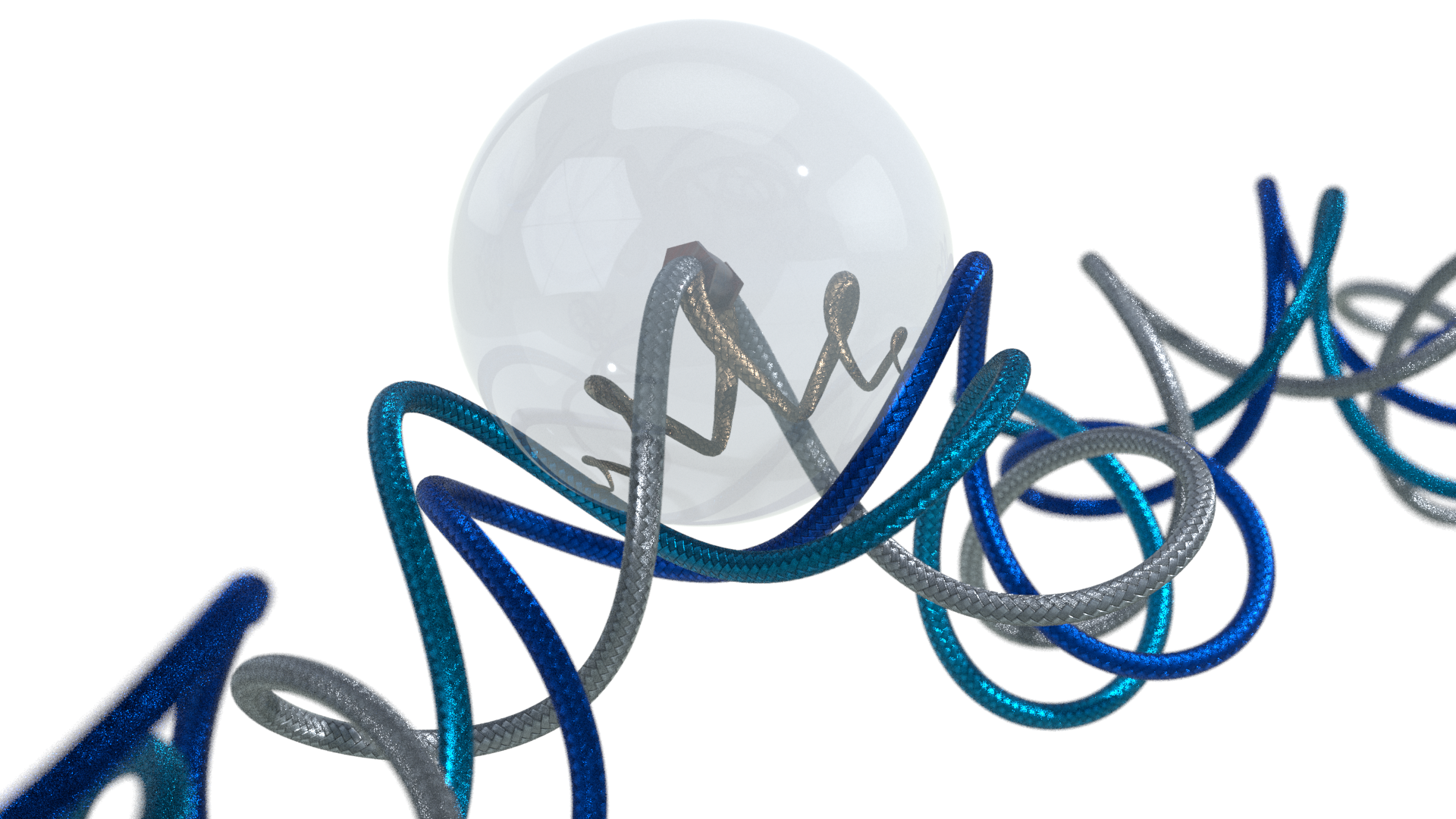}
  \caption{As the base point $x_0$ moves along a curve $\gamma$ (silver), the curves (gold) in its associated family for fixed complex $\lambda$ spiral towards points that trace out Darboux transforms $\eta_{\pm}$ (blue) of $\gamma$.
  \label{fig:Darboux}}
\end{figure}

For a finite gap curve $\gamma\in\mM$  the hierarchy of flows corresponds to the osculating flag of the Abel--Jacobi embedding of the spectral curve $\Sigma$  at the point over $\lambda=\infty$. On the other hand, Darboux transformations correspond to translations along secants of the Abel image of $\Sigma$ in its Jacobian $\Jac(\Sigma)$.

\section{History of elastic curves and Hamiltonian curve flows}\label{sec:history}

For the early history of elastic space curves we follow \cite{tjaden1991}, see also \cite{levien2008,truesdell1983}. In 1691 Jakob Bernoulli posed the problem of determining the shape of bent beams \cite{bernoulli1691}. It was his nephew Daniel Bernoulli who, in 1742, realized in a letter to Euler \cite{Bernoulli1742} that this problem amounts to minimizing $\int\kappa^2$ for the curve that describes the beam. Euler then classified in 1744 all planar elastic curves \cite{euler1744,bryant1986}. Lagrange started in 1811 to investigate elastic space curves, but he ignored the gradient of total torsion that in general has to be part of the variational functional. This was pointed out in 1844 by Bin{\'e}t \cite{binet1844}, who wrote the complete Euler--Lagrange equations and was able to solve them up to quadratures. Then in 1859 Kirchhoff \cite{kirchhoff1859} discovered that these Euler--Lagrange equations can be interpreted as the equations of motion of the Lagrange top, a fact that became famous as the {\em Kirchhoff analogy}. Finally, in 1885 Hermite solved the equations explicitly \cite{hermite1885} in terms of elliptic functions. 

In 1906 Max Born wrote his Ph.D.~thesis on the stability of elastic curves \cite{born1906}. More recently, elastic curves in spaces other than $\mathbb{R}^3$ have been explored \cite{singer2008}. The gradient flows of their variational functionals have been studied from the viewpoint of Geometric Analysis \cite{dziuk2002}. Planar critical points of elastic energy under an area constraint were also investigated \cite{arreaga2002,ferone2016}.

The history of Hamiltonian curve flows begins in 1906 with the discovery of the vortex filament equation by Da Rios \cite{daRios1906}, who was a Ph.D.~student of Levi-Civita. In 1932 Levi-Civita described what nowadays would be called the one-soliton solution \cite{levi1932}. The corresponding curves are indeed elastic loops already found by Euler, but  Da Rios and Levi-Civita were seemingly not aware of this fact. For further details on the history see \cite{ricca1996} and also \cite{ricca1991}. Under the name of {\em localized induction approximation} this equation is a standard tool in Fluid Dynamics \cite{saffman1992}.
Rigorous estimates indicate that this approximation to the 3D Euler equation is valid even over a short time \cite{jerrard2017}.
Recently it became possible to study knotted vortex filaments experimentally \cite{kleckner2013}. The symplectic form on the space of curves was found by Marsden and Weinstein \cite{marsden1983} in 1983, see also Chapter VI.3 of \cite{arnold1999}. Codimension two submanifolds in higher dimensional ambient spaces can be treated similarly \cite{khesin2013}. 

The relationship between vortex filaments and the theory of integrable systems was discovered in 1972 by Hasimoto. He showed the equivalence of the vortex filament flow with the nonlinear Schr{\"o}dinger equation \cite{hasimoto1972}. The non-linear Schr\"odinger hierarchy then led to the discovery of the hierarchy of flows for space curves \cite{langer1991,yasui1998,langer1999}. This made it possible to study in detail \cite{grinevich2000,calini2000,calini2004} those curves in $\R^3$ corresponding to finite gap solutions of the nonlinear Schr{\"o}dinger equation \cite{previato1985}. Since every other of the Hamiltonian curve flows preserves the planarity of curves, this allows for a self-contained approach to the mKdV hierarchy of flows on plane curves \cite{matsutani2016}.

In magnetohydrodynamics vortex filaments can carry a trapped magnetic field, in which case the total torsion becomes part of the Hamiltonian. As discussed above, the resulting flow is the helicity filament flow \cite{holm2004}.

Darboux transformations with purely imaginary $\lambda$ have been studied under the name {\em bicycle transformations} \cite{tabachnikov2017,bor2017}. Those with general complex $\lambda$ were investigated in \cite{pinkall2007} and in \cite{hoffmann2008}, where they were called B{\"a}cklund transformations.

\bibliographystyle{acm}
\bibliography{CurveFlows}

\end{document}